\documentclass[reqno]{amsart}
\pdfoutput=1

\usepackage{amsmath}
\usepackage{multicol}
\usepackage{mathrsfs}
\usepackage{amssymb}
\usepackage{amsthm}
\usepackage[top=95pt,bottom=80pt,left=85pt,right=85pt,footskip=35pt]{geometry}
\usepackage{xcolor}
\usepackage{enumerate}
\usepackage[all]{xy}
\xyoption{pdf}
\usepackage{tabularx}

\usepackage{makecell}
\usepackage{csquotes}
\usepackage{hyperref}
\usepackage{tikz}
\usetikzlibrary{calc}
\usetikzlibrary{intersections}
\usepackage{ctable}
\usepackage{fancyhdr,floatpag}  

 \floatpagestyle{fancy}

\numberwithin{equation}{section}

\theoremstyle{definition}

\newtheorem*{Thm}{Main Theorem}
\newtheorem*{Thm2}{Theorem~\ref{space-cor}}
\newtheorem*{Cor}{Corollary}
\newtheorem{Prop}{Proposition}[section]
\newtheorem{Theorem}[Prop]{Theorem}
\newtheorem{Lem}[Prop]{Lemma}
\newtheorem{Rmk}[Prop]{Remark}

\newcommand{\PP}{\mathbb{P}}

\newcommand{\QQ}{\mathbb{Q}}
\newcommand{\ZZ}{\mathbb{Z}}
\newcommand{\Oo}{\mathcal{O}}
\newcommand{\Ii}{\mathcal{I}}
\newcommand{\Uu}{\mathcal{U}}
\newcommand{\Kk}{\mathcal{K}}
\newcommand{\Mm}{\mathcal{M}}
\newcommand{\Tt}{\mathcal{T}}
\newcommand{\Cc}{\mathcal{C}}
\newcommand{\Dd}{\mathcal{D}}
\newcommand{\Ee}{\mathcal{E}}
\newcommand{\Gg}{\mathcal{G}}
\newcommand{\Hh}{\mathcal{H}}
\newcommand{\ee}{\overline{e}}
\renewcommand{\ggg}{\sigma}
\newcommand{\elll}{{l}}
\newcommand{\qq}{\overline{q}}
\newcommand{\Pic}{\operatorname{Pic}}

\newcommand{\Bl}{\mathrm{Bl}}
\newcommand{\Aut}{\operatorname{Aut}}
\newcommand{\Sym}{\operatorname{Sym}}

\begin{document}

\title{Bigness of the tangent bundle of a Fano threefold with Picard number two}
\author{Hosung Kim, Jeong-Seop Kim, and Yongnam Lee}

\address {Department of Mathematics\\
Changwon National University\\
20 Changwondaehak-ro, Uichang-gu\\
Changwon-si, Gyeongsangnam-do, 51140 Korea}
\email{hosungkim@changwon.ac.kr}

\address {School of Mathematics\\
Korea Institute for Advanced Study\\
85 Hoegiro, Dongdaemun-gu\\
Seoul, 02455 Korea}
\email{jeongseop@kias.re.kr}

\address {Center for Complex Geometry\\
Institute for Basic Science (IBS)\\
55 Expo-ro, Yuseong-gu\\ 
Daejeon, 34126 Korea, and 
\newline\hspace*{3mm} Department of Mathematical Sciences\\
KAIST\\
291 Daehak-ro, Yuseong-gu\\ 
Daejeon, 34141 Korea}
\email{ynlee@ibs.re.kr}


\begin{abstract}
In this paper, we study the positivity property of the tangent bundle $T_X$ of a Fano threefold $X$ with Picard number 2.
We determine the bigness of the tangent bundle of the whole 36 deformation types.
Our result shows that $T_X$ is big if and only if $(-K_X)^3\ge 34$.
As a corollary, we prove that the tangent bundle is not big when $X$ has a standard conic bundle structure with non-empty discriminant.
Our main methods are to produce irreducible effective divisors on $\PP(T_X)$ constructed from the total dual VMRT associated to a family of rational curves.
Additionally, we present some criteria to determine the bigness of $T_X$.
\end{abstract}

\maketitle

\section{Introduction}

Throughout this paper, we will work over the field of complex numbers.
Let $X$ be a smooth projective variety.
We say that the tangent bundle $T_X$ of $X$ is pseudoeffective (resp. big) if the tautological class $\Oo_{\PP(T_X)}(1)$ of the projectivized bundle $\PP(T_X)$ is pseudoeffective (resp. big). 

In general, it is difficult to give a numerical characterization of pseudoeffectivity or bigness of the tangent bundle, even in low dimension with low rank of the Picard group.
It has been shown by Hsiao \cite[Corollary~1.3]{Hsi} that the tangent bundle of a toric variety is big.
H\"oring, Liu, and Shao \cite[Theorem~1.4]{HLS} give a complete answer to del Pezzo surfaces: If $X$ is a del Pezzo surface of degree $d$,~then
\begin{itemize}
\item[(a)] $T_X$ is pseudoeffective if and only if $d \ge 4$, 
\item[(b)] $T_X$ is big if and only if $d\ge 5$.
\end{itemize}
Also in the paper \cite{HLS}, they solve these problems for del Pezzo threefolds.
In \cite{HL}, H\"oring and Liu consider Fano manifolds $X$ with Picard number one, and they prove that if $X$ admits a rational curve with trivial normal bundle and with big $T_X$, then $X$ is isomorphic to the del Pezzo threefold of degree five.
These all results indicate that assuming bigness should lead to strong restrictions on Fano manifolds, and lead us to consider naturally Fano threefolds with Picard number 2.
In this paper, we prove the following main theorem.
The details are in Table~\ref{the_table} in the introduction.

\begin{Thm}
We determine the bigness of the tangent bundle $T_X$ of whole 36 deformation types of Fano threefolds $X$ with Picard number 2.
In particular, the tangent bundle $T_X$ is big if and only if $(-K_X)^3 \ge \, 34$.
\end{Thm}

We show the main theorem by using the common property of all elements in the deformation family.
Therefore, the property of bigness of $T_X$ does not depend on deformations.

We also note that Fano threefolds of $(-K_X)^3 \ge \, 34$ (No. 26 -- 36 in \cite[Table~2]{MM}) have infinite automorphism groups \cite[Theorem~1.2]{PCS}.
As a corollary,  we give a complete answer for the bigness of the tangent bundle $T_X$ when $X$ has a standard conic bundle structure.

\begin{Cor}
Let $X$ be a Fano threefold with Picard number 2.
We suppose that $X$ has a standard conic bundle structure.
Then $T_X$ is not big if and only if $X$ has a standard conic bundle structure with non-empty discriminant.
\end{Cor}

Our method and explicit description of total dual VMRTs are also applicable to some general cases other than the case of Fano threefolds.
Besides our main theorem, we obtain the following.

\begin{Thm2}
Let $X$ be the blow-up of $\PP^3$ along a smooth nondegenerate curve $\Gamma$.
Assume that $\Gamma$ has at most a finite number of quadrisecant lines on $\PP^3$.
Then $T_X$ is big if and only if $\Gamma$ is a twisted cubic curve. 
\end{Thm2}

A Fano threefold $X$ is called primitive if it is not isomorphic to the blow-up of a Fano threefold along a smooth irreducible curve.
Due to Batyrev's classification of toroidal Fano threefolds \cite{Bat} and Mori-Mukai's classification of Fano threefolds with Picard number 2, No. 32 in \cite[Table~2]{MM} (a divisor on $\PP^2\times\PP^2$ of bidegree $(1, 1)$) is the only non-toric case which has a conic bundle structure with empty discriminant. In this case, $T_X$ is big (see Remark~\ref{MM:2-32}).
We note that if $X$ is a primitive Fano threefold with Picard number 2, then $X$ has a standard conic bundle structure \cite[Theorem~5]{MM}.

%Let $X$ be a Fano threefold with a standard conic bundle structure.
A conic bundle is a proper flat morphism $\pi: X\to S$ of nonsingular varieties such that it is of relative dimension 1 and the anticanonical divisor $-K_X$ is relatively ample.
A conic bundle $\pi: X\to S$ is called standard if for any prime divisor $D\subset S$, its inverse image $\pi^*(D)$ is irreducible.
If $X$ is a Fano threefold with Picard number 2 admitting a conic bundle $\pi: X\to S$, then  $S$ is the projective plane $\mathbb P^2$ and $\pi$ is standard.  
%Picard number 2 condition implies that $X$ has a conic bundle structure over the projective plane $\PP^2$.
Let $\pi: X\to \PP^2$ be a conic bundle structure over $\PP^2$ with the discriminant curve $\Delta\subseteq\PP^2$ of degree $d=\deg \Delta$.
By \cite[\S 1]{Isk} and \cite[Corollary~3.3.3, Corollary~3.9.1]{Prok}, $\Delta$ has only normal crossings in $\PP^2$ and $d\geq 3$.
Also, from \cite[Lemma 3.6]{Prok}, 
\(
b_3(X)=2b_2(\PP^2)+2b_2(X)-2b_2(\PP^2)+2p_a(\Delta)-4=2p_a(\Delta)-2=d^2-3d.
\)

Assume that $X$ is a Fano threefold with Picard number 2 which has a standard conic bundle structure with non-empty discriminant.
Then according to \cite[Table~2]{MM},
\(
3\leq d\leq 8
\,\,\, \text{and}\,\,\,
d\neq 7,
\)
and there are 9 deformation types (No. 2, 6, 8, 9, 11, 13, 18, 20 and 24 in \cite[Table~2]{MM}). 

\medskip

Let $\PP(T_X)$ be the projectivized bundle $\Pi: \PP(T_X)\to X$ of the tangent bundle $T_X$ of $X$.
We will denote by $\zeta=\Oo_{\PP(T_X)}(1)$ the tautological class of $\PP(T_X)$. Here, we use Grothendieck's notion for $\PP(T_X)$.
We note that $-K_{\PP(T_X)}=3\zeta$.
Our main strategy of the proof of the main theorem is to find two irreducible effective divisors $\breve\Cc_1$ and $\breve\Cc_2$ on $\PP(T_X)$ which are the total dual VMRTs associated to families of rational curves, and express a positive multiple of $\zeta$ as a combination of $[\breve\Cc_1]$, $[\breve\Cc_2]$ and $\alpha\cdot(\text{effective divisor})$ with $\alpha\in\ZZ$.

In particular, if $X$ has a standard conic bundle structure $\pi: X\to \PP^2$, we have a natural irreducible effective divisor $\breve\Cc$ on $\PP(T_X)$ induced from the fibers of $\pi: X\to \PP^2$ according to \cite[Corollary~2.13]{HLS}, and $[\breve\Cc]\sim \zeta+\Pi^*(K_X-\pi^*K_{\mathbb P^2}).$ To find irreducible effective divisors $\breve\Cc$ constructed from the total dual VMRT associated to a family of rational curves, we use explicit descriptions of Fano threefolds with Picard number 2 in \cite[Table~2]{MM}.
Especially, we describe irreducible effective divisors $\breve\Cc$ induced from the total dual VMRT associated to a family of rational curves when $X$ are imprimitive Fano threefolds.
Additionally, we provide some criteria to disprove the bigness of $T_X$ when $X$ is the blow-up of a smooth curve on $\PP^3$ or a quadric hypersurface $Q$ in $\PP^4$ or the quintic del Pezzo threefold $V_5$.

\medskip

The organization of the paper is as follows.

In Section 2, we briefly introduce the theory related to the total dual VMRT and present some criteria to disprove the bigness of $T_X$.
Proposition~\ref {three_family} provides a criterion to disprove the bigness of $T_X$ by making use of two rational curves on $X$ not belonging to a given family which associates a total dual VMRT on $\PP(T_X)$.
In Proposition~\ref{fibration} we prove that $T_X$ is not big when $X$ has a del Pezzo surface of degree $d\leq 4$ fibration.

Section 3 presents some criteria to determine the bigness of $T_X$ when $X$ is an imprimitive Fano threefold, i.\,e., $X$ is isomorphic to the blow-up $f: X=\Bl_\Gamma Z\to Z$ of a Fano threefold $Z$ along a smooth curve $\Gamma$.
We first observe a relation between the bigness of $T_X$ and $T_Z$, and then investigate the cases when $Z$ is $\PP^3$ (Remark~\ref{divisor_from_projective_space}) or a quadric hypersurface $Q$ in $\PP^4$ (Proposition~\ref{quadric}) or the quintic del Pezzo threefold $V_5$ (Proposition~\ref{V5}).

Section 4 describes irreducible effective divisors $\breve\Cc$ on $\PP(T_X)$ induced from the total dual VMRT associated to a family of rational curves when $X$ is isomorphic to the blow-up of a smooth curve $\Gamma$ on $\PP^3$ (Proposition~\ref{space} and Proposition~\ref{space2}) or a quadric hypersurface $Q$ in $\PP^4$ (Proposition~\ref{quadric2}).
We consider the family of the secant lines of $\Gamma$ on $\PP^3$, and the family of the lines meeting at one point of $\Gamma$ on $\PP^3$ and $Q$.
Our explicit description of total dual VMRTs in Proposition~\ref{space} and Proposition~\ref{space2} gets Theorem~\ref{space-cor} as a corollary.

In Section 5, we treat mainly Fano threefolds $X$ with Picard number 2 which admit a standard conic bundle structure with non-empty discriminant.
%Remark~\ref{expectation} shows a relation between the coefficients of the total dual VMRT induced from the secant lines of the blow-up locus curve $\Gamma$ on $\PP^3$, and the degrees of the tangential surface and the edge surface of $\Gamma$. 

\smallskip

In this present paper, we determine the bigness of the tangent bundle $T_X$ of Fano threefolds $X$ with Picard number $2$.
We also study the pseudoeffectivity of $T_X$ and expect our methods to give some answer on Fano threefolds with higher Picard numbers.
Based on the result of Fano threefolds with Picard number one \cite[Corollary~1.2]{HL} and our result for Fano threefolds with Picard number two, we expect the anti-canonical degree determines the bigness of $T_X$ for Fano threefolds.
Hence we raise bravely the following conjecture.

\smallskip

\noindent{\bf Conjecture.} Let $X$ be a Fano threefold. Then there is a constant $C_0$ depending only on the Picard number of $X$ such that $T_X$ is big if and only if $(-K_X)^3\ge C_0$.

\smallskip

The following table summarizes the bigness of the tangent bundle $T_X$ of Fano threefolds $X$ with Picard number 2.
In the table, the numbers in the first columns correspond to the deformation types in \cite[Table~2]{MM}, and those in the last columns indicate the remark or subsection of each case.
As a by-product of the proof, we can see that the tangent bundle of a general element in the families of No. 21, 23, and 24 is $\QQ$-effective and not big. 

\begin{table} \footnotesize \centering
\begin{tabularx}{\textwidth}{c|c|X|c}
\specialrule{.125em}{0em}{0em}
No.
&
$(-K_X)^3$
&
\centering $X$
&
Big $T_X$ \\
\hline
\begin{tabular}{@{}c@{}}
1, 3, 5,\\ 10, 16, 19
\end{tabular} &
\begin{tabular}{@{}c@{}}
4, 8, 12,\\ 16, 22, 26
\end{tabular} &
blow-ups of a smooth del Pezzo threefold $V_i$ of degree $i$ where $i=1,\,2,\,3,\,4$
&
$\times$ (\ref{blow-ups_of_del_pezzo_manifold})\\
\hline
2 & 
6 &
a double cover of $\PP^2\times\PP^1$ whose branch locus is a divisor of bidegree $(4,2)$
& $\times$ (\ref{MM:2-2})\\
\hline
4 & 
10 &
the blow-up of $\PP^3$ with center an intersection of two cubics
&
$\times$ (\ref{blow-ups_of_projective_space}) \\
\hline
\makecell[ct]{\vspace{-0.8em}\\ 6} & 
\makecell[ct]{\vspace{-0.8em}\\ 12} & 
\makecell[Xt]{
(6.a) a divisor on $\PP^2\times\PP^2$ of bidegree $(2,2)$\\ 
(6.b) a double cover of $W$ whose branch locus is a member of $|-K_W|$
}
&
\makecell[ct]{\vspace{-0.8em}\\$\times$ (\ref{MM:2-6})} \\
\hline
7 & 
14 &
the blow-up of smooth quadric threefold $Q\subseteq\PP^4$ with center an intersection of two members of $|\Oo_Q(2)|$
&
$\times$ (\ref{del_Pezzo_fibration}) \\
\hline
8 &
14 &
a double cover of the blow-up $V_7$ of $\PP^3$ at a point whose branch locus is a member $B$ of $|-K_{V_7}|$ such that
(8.a) $B\cap D$ is smooth
(8.b) $B\cap D$ is reduced but not smooth,
where $D$ is the exceptional divisor of the blow-up $V_7\to\PP^3$
& $\times$ (\ref{MM:2-8})\\
\hline
9 &
16 &
the blow-up of $\PP^3$ with center a curve of degree $7$ and genus $5$ which is an intersection of cubics
&
$\times$ (\ref{MM:2-9})\\
\hline
11 &
18 &
the blow-up of a smooth cubic threefold $V_3\subseteq \PP^4$ with center a line on it
&
$\times$ (\ref{MM:2-11})\\
\hline
12 &
20 &
the blow-up of $\PP^3$ with center a curve of degree $6$ and genus $3$ which is an intersection of cubics
& 
$\times$ (\ref{blow-ups_of_projective_space})\\
\hline
13 &
20 &
the blow-up of a smooth quadric threefold $Q\subseteq\PP^4$ with center a curve of degree $6$ and genus $2$
&
$\times$ (\ref{MM:2-13})\\
\hline
14 &
20 &
the blow-up of a smooth del Pezzo threefold $V_5\subseteq\PP^6$ of degree $5$ with center an elliptic curve which is an intersection of two hyperplane sections
&
$\times$ (\ref{MM:2-14}) \\
\hline
15 &
22 &
the blow-up of $\PP^3$ with center an intersection of a quadric $A$ and a cubic $B$ such that (15.a) $A$ is smooth (15.b) $A$ is reduced but not smooth
&
$\times$ (\ref{blow-ups_of_projective_space}) \\
\hline
17 &
24 &
the blow-up of a smooth quadric threefold $Q\subseteq\PP^4$ with center an elliptic curve of degree $5$ on it
&
$\times$ (\ref{blow-ups_of_projective_space})\\
\hline
18 &
24 &
a double cover of $\PP^2\times \PP^1$ whose branch locus is a divisor of bidegree $(2,2)$
&
$\times$ (\ref{MM:2-18})\\
\hline
20 &
26 &
the blow-up of a smooth del Pezzo threefold $V_5\subseteq\PP^6$ with center a twisted cubic on it
&
$\times$ (\ref{MM:2-20})\\
\hline
21 &
28 &
the blow-up of a smooth quadric threefold $Q\subseteq\PP^4$ with center a twisted quartic, a smooth rational curve of degree $4$ which spans $\PP^4$, on it
&
$\times$ (\ref{blow-ups_of_quadric_threefold})\\
\hline
22 &
30 &
the blow-up of $V_5\subseteq\PP^6$ with center a conic on it
&
$\times$ (\ref{blow-ups_of_projective_space})\\
\hline
23 &
30 &
the blow-up of a smooth quadric threefold $Q\subseteq\PP^4$ with center an intersection of $A\in|\Oo_Q(1)|$ and $B\in|\Oo_Q(2)|$ such that (23.a) $A$ is smooth (23.b) $A$ is not smooth
&
$\times$ (\ref{blow-ups_of_quadric_threefold2}) \\
\hline
24 &
30 &
a divisor on $\PP^2\times \PP^2$ of bidegree $(1,2)$
&
$\times$ (\ref{MM:2-24})\\
\hline
25 &
32 &
the blow-up of $\PP^3$ with center an elliptic curve which is an intersection of two quadrics 
&
$\times$ (\ref{blow-ups_of_projective_space})\\
\hline
26 &
34 &
the blow-up of a smooth del Pezzo threefold $V_5\subseteq\PP^6$ of degree $5$ with center a line on it
&
$\bigcirc$ (\ref{MM:2-26_and_31}) \\
\hline
27 &
38 &
the blow-up of $\PP^3$ with center a twisted cubic
&
$\bigcirc$ (\ref{divisor_from_projective_space}) \\
\hline
28 &
40 &
the blow-up of $\PP^3$ with center a plane cubic
&
$\bigcirc$ (\ref{divisor_from_projective_space}) \\
\hline
29 &
40 &
the blow-up of a smooth quadric threefold $Q\subseteq\PP^4$ with center a conic on it
&
$\bigcirc$ (\ref{blow-ups_of_quadric_threefold2}) \\
\hline
30 &
46 &
the blow-up of $\PP^3$ with center a conic
&
$\bigcirc$ (\ref{divisor_from_projective_space}) \\
\hline
31 &
46 &
the blow-up of a smooth quadric threefold $Q\subseteq \PP^4$ with center a line on it
&
$\bigcirc$ (\ref{MM:2-26_and_31}) \\
\hline
32 &
48 &
a divisor on $\PP^2\times \PP^2$ of bidegree $(1,1)$
&
$\bigcirc$ (\ref{MM:2-32}) \\
\hline
33 -- 36 &
54 -- 62 &
smooth toric varieties
&
$\bigcirc$ \\
\specialrule{.125em}{0em}{0em}
\end{tabularx}

\vspace{1em}
\caption{The bigness of $T_X$ for Fano threefolds $X$ with Picard number 2.}
\label{the_table}
\end{table}

\medskip

\noindent
{\bf Acknowledgements.}
The authors are supported by the Institute for Basic Science IBS-R032-D1.
J.-S. Kim is also supported by the NRF grant funded by the Korea government (MSIT) RS-2024-00349592.
We would like to thank Feng Shao for explaining to us the results of his paper with H\"oring and Liu, and Andreas H\"oring and Jie Liu for some useful comments.
We would also like to thank the referee for the very careful reading and suggestions to improve the paper.

\section{Total Dual VMRTs and some Criteria to Disprove Bigness}

In this section, we briefly introduce the theory related to the total dual VMRT, which is firstly introduced by \cite{HR} in a study of Hecke curves on the moduli of vector bundles on a curve.
Later, \cite{OSW} generalizes the theory to the case of minimal rational curves, and \cite{HLS} develops explicit formulas in the case where a variety has zero-dimensional VMRTs in a study of the tangent bundles of del Pezzo manifolds.

\medskip

Let $X$ be a smooth projective variety. Let $\text{RatCurves}^n(X)$ be the normalized space of rational curves on $X$ (see \cite[Chapter II]{Kol}).
We mean by a \emph{family of rational curves} on $X$ an irreducible component $\Kk$ of $\text{RatCurves}^n(X)$.

We say  a rational curve $\ell$ on $X$ is \emph{unbendable} if its normalization $\nu:\PP^1\to \ell\subseteq X$ satisfies
\[
\nu^*T_X\cong \Oo_{\PP^1}(2)\oplus {\Oo_{\PP^1}(1)}^{\oplus r}\oplus {\Oo_{\PP^1}}^{\oplus s}
\ \ 
\text{for some $r,\,s\geq 0$ with $r+s+1=\dim X$}.
\] An irreducible component  $\Kk$ of  $\text{RatCurves}^n(X)$ is called a \emph{family of unbendable rational curves} on $X$ if its general member $[\ell]\in\Kk$ is unbendable. 

In the rest of this section we consider a family $\Kk$ of unbendable rational curves on $X$. 
Let $q: \Uu\to\Kk$ be the normalization of the universal family and $e:\Uu\to X$ be the evaluation morphism.
%\vspace{-0.3em} %%%%%%%%%%%%%%%%%%%%%%%%%%%%%%%%%%%%%%%%%%%%%%%%%%
\[\xymatrix @R=2pc @C=3pc {
\Uu \ar[r]^{e} \ar[d]_{q} & X \\
\Kk
}\]
It is known that the evaluation morphism $e:\Uu\to X$ is dominant. 
We denote by $\Kk_x$ the normalization of $q(e^{-1}(x))\subseteq\Kk$.
For a general point $x\in X$, there exists a rational map $\tau_x: \Kk_x\dashrightarrow \PP(\Omega_X|_x)$, which is called the \emph{tangent map}, sending a curve which is smooth at $x$ to its tangent direction at $x$, and we define the \emph{variety of the minimal rational tangents} (\emph{VMRT}, for short) $\Cc_x$ of $\Kk$ at $x$ to be the closure of the image of $\tau_x$ in $\PP(\Omega_X|_x)$.

For an unbendable rational curve $\ell$ on $X$, a \emph{minimal section} of $\ell$ is defined by a rational curve $\widetilde{\ell}$ on $\PP(T_X)$ with the normalization 
%\vspace{-0.3em} %%%%%%%%%%%%%%%%%%%%%%%%%%%%%%%%%%%%%%%%%%%%%%%%%%
\[
\widetilde{\nu}: \PP^1\to \widetilde{\ell}\subseteq \PP(T_X|_\ell) \subseteq\PP(T_X)
\]
associated  to a trivial quotient $\nu^*T_X\to\Oo_{\PP^1}$.

Unless $X\cong\PP^n$, a general member $[\ell]\in\Kk$ is unbendable with $s\neq 0$, so there exists a minimal section $\widetilde{\ell}$ on $\PP(T_X)$.
Let $\widetilde{\Kk}$ be a family of rational curves on $\PP(T_X)$ which contains $[\widetilde{\ell}]$.
Then, for the normalized universal family $\widetilde{q}:\widetilde{\Uu}\to\widetilde{\Kk}$ and its evaluation morphism $\widetilde{e}:\widetilde{\Uu}\to \PP(T_X)$, we have the following commutative diagram (see \cite[Section 4]{OSW}).
%\vspace{-0.3em} %%%%%%%%%%%%%%%%%%%%%%%%%%%%%%%%%%%%%%%%%%%%%%%%%%
\begin{equation*}
\xymatrix @R=2pc @C=3pc {
\widetilde{\Kk} \ar@{<-}[r]^{\widetilde{q}} \ar[d]
& \widetilde{\Uu} \ar[r]^{\widetilde{e}}\ar[d]
& \PP(T_X) \ar[d]^\Pi \\
\Kk \ar@{<-}[r]_{q} & \Uu \ar[r]_{e} & X
}
\end{equation*}

We denote by $\breve\Cc$ the closure of the image $\widetilde{e}(\widetilde{\Uu})$ in $\PP(T_X)$.
The next proposition is essentially proved in {\cite[Proposition~5]{OSW}} under the assumption that $\Kk$ is \emph{locally unsplit}, i.\,e., $\Kk_x$ is proper for general $x\in X$.

\begin{Prop}[cf. {\cite[Proposition~5]{OSW}}]
Let $\Kk$ be a family of unbendable rational curves on $X$.
Then, for general $x\in X$, $\breve\Cc|_x\subseteq\PP(T_X|_x)$ is the projectively dual variety of the VMRT $\Cc_x\subseteq\PP(\Omega_X|_x)$.
\end{Prop}

Due to the proposition above, $\breve\Cc$ is called the \emph{total dual VMRT} associated to $\Kk$.
As there exists a member $[\widetilde{\ell}]\in\widetilde{\Kk}$ which satisfies $\zeta.\widetilde{\ell}=0$, we can observe that $\breve\Cc$ is 
\[
\label{condition}
\tag{$\dagger$}
\text{dominated by a family of rational curves $\widetilde{\ell}$ satisfying $\zeta.\widetilde{\ell}=0$ on $\PP(T_X)$}.
\]
Notice that for such $\widetilde{\ell}$ with $\zeta.\widetilde{\ell}=0$, $\Pi(\widetilde{\ell})$ is a curve on $X$ as $\zeta$ is relatively ample on $\PP(T_X)$ over $X$.

\medskip

Our main strategy to disprove the bigness of $T_X$ is to find an effective divisor $\Dd$ on $\PP(T_X)$ which attains the property \eqref{condition} and apply the following lemma.

\begin{Lem}\label{repeat}
Let $\Dd=\sum_{i=1}^n
{\Dd_i}$ for some irreducible and reduced effective divisors $\Dd_i$ on $\PP(T_X)$.
Assume that $\Dd_i$ satisfies \eqref{condition} for all $i=1,\,2,\,\ldots,\,n$.
If $\Dd\sim k\zeta+\Pi^*H$ for some $k>0$ and effective divisor $H$ on $X$ (possibly $H=0$), then $\zeta$ is not big on $\PP(T_X)$.
\end{Lem}

\begin{proof}
Suppose that $\zeta$ is big.
We can write $\Dd_i=k_i\zeta+\Pi^*H_i$ for some $k_i\geq 0$ and divisor $H_i$ on $X$.
Let $A$ be an ample divisor on $X$ such that $A+\sum_{i=1}^{l} H_i$ is also ample for all $1\leq l\leq n$.
Then there exists the smallest integer $m>0$ such that $m\zeta-\Pi^*A$ is effective by Kodaira's lemma.
Let $\Ee_0\sim m\zeta-\Pi^*A$ be an effective divisor.
For a general curve $\widetilde{\ell_1}$ in the family whose members cover the divisor $\Dd_1$, we have
\[
\Ee_0.\widetilde{\ell_1}
=(m\zeta-\Pi^*A).\widetilde{\ell_1}
=-A.\Pi(\widetilde{\ell_1})<0.
\]
Thus $\Ee_0-\Dd_1\geq 0$, and we can find an effective divisor $\Ee_1=\Ee_0-\Dd_1$.
For a general curve $\widetilde{\ell_2}$ in the family whose members cover the divisor $\Dd_2$, we also have
\[
\Ee_1.\widetilde{\ell_2}
=((m-k_1)\zeta-(\Pi^*A+\Pi^*H_1)).\widetilde{\ell_2}
=-(A+H_1).\Pi(\widetilde{\ell_2})<0.
\]
as $A+H_1$ is ample by the assumption.
Thus we are again able to take an effective divisor $\Ee_2=\Ee_1-\Dd_2$.
By iterating the process, we arrive to obtain an effective divisor
\[
\Ee_n
=\Ee_{n-1}-\Dd_n
=\cdots
=\Ee_0-\Dd
=\left(m-\sum_{i=1}^n k_i\right)\zeta-\left(\Pi^*A+\sum_{i=1}^n \Pi^*H_i\right)
=(m-k)\zeta-\Pi^*A-\Pi^*H
\]
on $\PP(T_X)$.
In particular, $(m-k)\zeta-\Pi^*A$ is effective, and it contradicts the minimality of $m$.
\end{proof}

Now we want to find the explicit linear class of the total dual VMRT.
Let $\Kk$ be a family of {\color{black} unbendable} rational curves on $X$.
Let $\overline{\Kk}$ be the normalization of the closure of $\Kk$ considered as a subscheme of $\text{Chow}(X)$ and $\overline{q}: \overline{\Uu}\to \overline{\Kk}$ be the normalization of the universal family over $\overline{\Kk}$ which associates the evaluation morphism $\overline{e}: \overline{\Uu}\to X$.

\[\xymatrix @R=2pc{
\PP(\overline{e}^*T_X) \ar[r] \ar[d] &
\PP(T_X) \ar[d]^\Pi \\
\overline{\Uu} \ar[r]_{\overline{e}} \ar[d]_{\overline{q}} &
X\\
\overline{\Kk}
}\]

We denote by $\Tt_{\overline{\Uu}/\overline{\Kk}}$ the dual $\Omega_{\overline{\Uu}/\overline{\Kk}}^\vee\to T_{\overline\Uu}$.
Then there exists the exact sequence 
\[
0
\to \Tt_{\overline{\Uu}/\overline{\Kk}}
\to \overline{e}^*T_X
\to \Gg
\to 0
\]
for some coherent sheaf $\Gg$ on $\overline{\Uu}$.

\begin{Prop}[{\cite[Corollary~2.10, Remark~2.12, Corollary~2.13]{HLS}}]
\label{conic_bundle_formula}
Let $X$ be a smooth projective variety and $\Kk$ be a family of unbendable rational curves which is locally unsplit.
Assume that $\Kk$ has zero-dimensional VMRTs and $\Gg$ is locally free in codimension~$1$.
Then the linear class of the total dual VMRT $\breve\Cc$ on $\PP(T_X)$ is explicitly given as follows.
\[
[\breve\Cc]
\sim \deg(\overline{e})\cdot\zeta-\Pi^*\overline{e}_*(c_1(\Tt_{\overline{\Uu}/\overline{\Kk}}))
\]
Moreover, the condition on $\Gg$ is verified when $\Tt_{\overline{\Uu}/\overline{\Kk}}$ is locally free and $\overline{e}_{[\ell]}: \overline{\Uu}_{[\ell]}\to X$ is immersed for general ${[\ell]}\in\overline{\Kk}$.
In particular, if $X$ admits a conic bundle structure $\pi: X\to Y$ over a smooth projective variety $Y$, then the total dual VMRT $\breve\Cc$ on $\PP(T_X)$ associated to the fibers of $\pi$ satisfies
\[
[\breve\Cc]
\sim \zeta+\Pi^*(K_X-\pi^*K_Y).
\]
\end{Prop}

\begin{Rmk}
Let $\pi: X=\PP(E)\to \PP^2$ be the ruled variety associated to a vector bundle $E$ of rank $2$ on $\PP^2$ whose normalization has negative degree.
That is, $H^0(E\otimes L)\neq 0$ for some line bundle $L$ with $\det (E\otimes L)< 0$.
After normalization, we may assume that $H^0(E)\neq 0$ and $\det E=\Oo_{\PP^2}(-a)$ for some $a>0$.
Then, by Proposition~\ref{conic_bundle_formula}, we obtain an effective divisor
\[
[\breve\Cc]
\sim\zeta+\Pi^*(K_X-\pi^*K_{\PP^2})
=\zeta+\Pi^*(-2H-ah)
\]
on $\PP(T_X)$ for $h=\pi^*\Oo_{\PP^2}(1)$ and the tautological class $H=\Oo_{\PP(E)}(1)$.

Then we can write $\zeta$ by a positive linear combination
\[
\zeta
\sim[\breve\Cc]+2\Pi^*H+a\Pi^*h
\]
of \emph{three} effective divisors on $\PP(T_X)$ which are linearly independent in $\text{N}^1(\PP(T_X))$.
As a multiple of $\zeta$ lies in the interior of the cone $\overline{\text{Eff}}(\PP(T_X))$ of effective divisors, $\zeta$ is big on $\PP(T_X)$.
That is, $T_X$ is big.

However, when $X$ is a Fano threefold given as above (No. 35 and 36 in \cite[Table~2]{MM}), the bigness of $T_X$ follows from \cite{Hsi} as such Fano threefolds are toric.
\end{Rmk}

\begin{Rmk}
\label{MM:2-32}
Let $\pi: X=\PP(T_{\PP^2})\to \PP^2$ be the ruled variety associated to the tangent bundle $T_{\PP^2}$ of $\PP^2$.\linebreak
This is the Fano threefold $X$ isomorphic to a $(1,1)$-divisor on $\PP^2\times \PP^2$ (No. 32 in \cite[Table~2]{MM}).
Then $X$ has two $\PP^1$-fibration structures $\pi_i: X\to \PP^2$ for $i=1,\,2$. 
\[\xymatrix @R=2pc @C=2pc {
X \ar[d]_{\pi_1} \ar[r]^{\pi_2} & \PP^2 \\ \PP^2
}\]
We denote by $h_i=\pi_i^*\Oo_{\PP^2}(1)$.
Then $-K_X\sim 2h_1+2h_2$.

Let $\breve\Cc_i$ be the total dual VMRT on $\PP(T_X)$ associated to the family of fibers of $\pi_i: X\to\PP^2$ for $i=1,\,2$.
Then, by Proposition~\ref{conic_bundle_formula}, we have
\begin{align*}
[\breve\Cc_1]
&\sim\zeta+\Pi^*(K_X-\pi_1^* K_{\PP^2})
=\zeta+\Pi^*((-2h_1-2h_2)-(-3h_1))
=\zeta+\Pi^*h_1-2\Pi^*h_2,\\
[\breve\Cc_2]
&\sim\zeta+\Pi^*(K_X-\pi_2^* K_{\PP^2})
=\zeta+\Pi^*((-2h_1-2h_2)-(-3h_2))
=\zeta-2\Pi^*h_1+\Pi^*h_2.
\end{align*}
So we can write a multiple of $\zeta$ by a positive linear combination of effective divisors on $\PP(T_X)$ as
\[
2\zeta
\sim[\breve\Cc_1]+[\breve\Cc_2]+\Pi^*(h_1+h_2).
\]
Thus $\zeta$ is big on $\PP(T_X)$.
That is, $T_X$ is big.
\end{Rmk}

In the case where a Fano threefold $X$ has Picard number 2, the following proposition provides a~criterion to disprove the bigness of $T_X$ by making use of two rational curves on $X$ not belonging to a given family which associates a total dual VMRT on $\PP(T_X)$.

\begin{Prop}
\label{three_family}
Let $X$ be a smooth projective threefold with $\Pic(X)\cong\ZZ H_1\oplus\ZZ H_2$ for some effective divisors $H_1$ and $H_2$.
Let $\breve\Cc$ be the total dual VMRT associated to a family $\Kk$ of {\color{black} unbendable} rational curves on $X$.
Assume that $\breve\Cc$ is a divisor on $\PP(T_X)$.
Moreover, assume that there exist two smooth rational curves $\ell_1$ and $\ell_2$ on $X$ such that
\begin{itemize}
\item
for each $i=1,\,2$, there exists a point $x_i\in \ell_i$ where the VMRT $\Cc_{x_i}\subseteq\PP(\Omega_X|_{x_i})$ of $\Kk$ at $x_i$ is defined and $\breve\Cc|_{x_i}$ is the projective dual of $\Cc_{x_i}$.
Moreover,
\vspace{-0.3em}
\[
\tau_{\ell_i, x_i}\not\in\Cc_{x_i}
\]
for the tangent direction $\tau_{\ell_i,x_i}\in\PP(\Omega_X|_{x_i})$ of $\ell_i$ at $x_i$,

\item 
the normal bundle of $\ell_i$ in $X$ is given by
\vspace{-0.3em}
\begin{align*}
N_{\ell_i|X}\cong\Oo_{\PP^1}(a_{i1})\oplus\Oo_{\PP^1}(a_{i2})\ 
\text{for some $a_{i2}\leq a_{i1}\leq 0$},
\end{align*}

\item
and the intersection numbers satisfy
\vspace{-0.3em}
\[
H_1.{\ell_1}=0,\qquad
H_2.{\ell_1}>0,\qquad
H_1.{\ell_2}>0,\qquad
H_2.{\ell_2}=0.
\]
\end{itemize}
Then $T_X$ is not big.
\end{Prop}

\begin{proof}
Let 
\[
[\breve\Cc]
\sim k\zeta+b_1\Pi^*H_1+b_2\Pi^*H_2
\]
for some $k>0$.
Once we show that $b_i\geq 0$ for $i=1,\,2$, the assertion follows from Lemma \ref{repeat} because $\breve\Cc$ satisfies \eqref{condition}.

Let $x=x_1$ be the point where the first condition holds for $\ell=\ell_1$.
Let $\Hh=\PP(N_{\ell|X}|_{x})\subseteq\PP(T_X|_{x})$ be the hyperplane projectively dual to the point $\tau_{\ell,x}\in\PP(\Omega_X|_{x})$.
Then, from the first condition, we have
\[
\emptyset\not=\Hh\backslash\,\breve\Cc|_{x}\subseteq\PP(T_X|_{x}).
\]
Hence there exists a section $\widetilde{\ell}\subseteq\PP(T_X|_{\ell})\subseteq\PP(T_X)$ corresponding to a quotient $T_X|_{\ell}\to\Oo_{\PP^1}(a)$ for $a=a_{11}$ in a factor of $N_{{\ell}|X}$ passing through a point $w\in\Hh\backslash\breve\Cc|_{x}$ over $x$, i.\,e., $w=\widetilde{\ell}|_{x}=\widetilde{\ell}\cap\PP(T_X|_{x})$ (see Figure~\ref{fig:three_family}).
Thus we have $\widetilde{\ell}|_{x}\cap\breve\Cc|_{x}=\emptyset$ on $\PP(T_X|_{x})$, and it implies that $\widetilde{\ell}$ is not contained in $\breve\Cc$.

From the second condition, we can observe that every section $\widetilde{\ell}\subseteq\PP(T_X|_{\ell})\subseteq\PP(T_X)$ which corresponds to a quotient $T_X|_{\ell}\to\Oo_{\PP^1}(a)$ in a factor of $N_{{\ell}|X}$ satisfies $\zeta.\widetilde{\ell}=a=a_{11}\leq 0$ in $\PP(T_X)$.
Therefore, for such $\widetilde{\ell_1}=\widetilde{\ell}$, we have
\[
b_2(H_2.\ell_1)
\geq k(\zeta.\widetilde{\ell_1})+b_1(\Pi^*H_1.\widetilde{\ell_1})+b_2(\Pi^*H_2.\widetilde{\ell_1})
=[\breve\Cc].\widetilde{\ell_1}
\geq 0
\]
as $\widetilde{\ell_1}$ is not contained in $\breve\Cc$.
Thus $b_2\geq 0$.
By the same argument, we can show that $b_1\geq 0$.
\end{proof}

\begin{figure}[ht]
\vspace{-0.5em}
\begin{tikzpicture}[scale=1]

\pgfmathsetmacro{\op}{0.5}
\pgfmathsetmacro{\wi}{4}
\pgfmathsetmacro{\he}{1.2}
\pgfmathsetmacro{\ti}{1}
\pgfmathsetmacro{\centerx}{0.5}
\pgfmathsetmacro{\centery}{2.2}
\pgfmathsetmacro{\radius}{1.3}
\pgfmathsetmacro{\leng}{0.5}
\pgfmathsetmacro{\cosi}{0.866}
\pgfmathsetmacro{\si}{0.5}

\draw[thick] (0,0)--(\wi,0)--(\wi+\ti,\he)--(\ti,\he)--cycle;

\draw[draw=none, fill=gray!50!white, fill opacity=0.3] (\wi/2+\ti/2,\he/2)--(\centerx+\radius,\centery+0.3)--(\centerx-0.3,\centery-\radius)--cycle;

\draw[color=gray, fill=white, fill opacity=1] (\centerx,\centery) circle (\radius);
\begin{scope}
    \draw[draw=none, clip] (\centerx,\centery) circle (\radius);

    \node[circle,fill=black!50!black,inner sep=0pt,minimum size=3pt] at (\centerx+\radius*\leng*0,\centery-\radius*\leng) {};

    \node[circle,fill=darkgray,inner sep=0pt,minimum size=3pt] at (\centerx+\radius*\leng*\cosi,\centery+\radius*\leng*\si) {};

    \node[circle,fill=darkgray,inner sep=0pt,minimum size=3pt] at (\centerx-\radius*\leng*\cosi,\centery+\radius*\leng*\si) {};
\end{scope}

\draw[draw=none, fill=gray!50!white, fill opacity=0.3] (\wi/2+\ti/2,\he/2)--(\wi+\ti-\centerx-\radius,\centery+0.3)--(\wi+\ti-\centerx+0.3,\centery-\radius)--cycle;

\draw[color=gray, fill=white, fill opacity=1] (\wi+\ti-\centerx,\centery) circle (\radius);
\begin{scope}
    \draw[draw=none, clip] (\wi+\ti-\centerx,\centery) circle (\radius);

    \draw[thick,color=darkgray,shorten >=-3cm,shorten <=-3cm] (\wi+\ti-\centerx+\radius*\leng*\cosi,\centery+\radius*\leng*\si)--(\wi+\ti-\centerx+\radius*\leng*0,\centery-\radius*\leng);
    
    \draw[thick,color=darkgray,shorten >=-3cm,shorten <=-3cm] (\wi+\ti-\centerx-\radius*\leng*\cosi,\centery+\radius*\leng*\si)--(\wi+\ti-\centerx+\radius*\leng*0,\centery-\radius*\leng);
    
    \draw[very thick,color=black!50!black,shorten >=-3cm,shorten <=-3cm] (\wi+\ti-\centerx-\radius*\leng*\cosi,\centery+\radius*\leng*\si)--(\wi+\ti-\centerx+\radius*\leng*\cosi,\centery+\radius*\leng*\si);
    
    \node[circle,fill=black,inner sep=0pt,minimum size=3.5pt] at (\wi+\ti-\centerx+0.25,\centery+\radius*\leng*\si) {};
\end{scope}

\draw[color=darkgray, thick] (\wi*0.25+\ti*0.25,\he*0.8) .. controls (\wi*0.5+\ti*0.5,\he*0.55) .. (\wi*0.6+\ti*0.6,\he*0.2);

\draw[color=darkgray, thick] (\wi*0.75+\ti*0.75,\he*0.8) .. controls (\wi*0.5+\ti*0.5,\he*0.55) .. (\wi*0.7*\ti*0.7,\he*0.2);

\draw[very thick, color=black!50!black] (0.5+\ti/2,\he/2)--(-0.5+\wi+\ti/2,\he/2);

\node[circle,fill=black,inner sep=0pt,minimum size=3pt] at (\wi/2+\ti/2,\he/2) {};

\node at (\wi/2,-0.3) {\small $X$};

\node at (-1.4,1.2) {\small $\PP(\Omega_X|_x)$};
\node at (6.4,1.2) {\small $\PP(T_X|_x)$};

\node at (\wi*0.98,\he*0.25) {\small $\ell$};
\node at (\wi/2+\wi*0.12,\he*0.25) {\small $x$};

\node at (\centerx,\centery+\radius*\leng*\si+0.45) {\small {\color{darkgray}$\Cc_x$}};
\node at (\centerx,\centery-\radius*0.2) {\small {\color{black!50!black}$\tau_{\ell,x}$}};

\node at (\wi+\ti-\centerx+0.2,\centery+\radius*\leng*\si+0.45) {\small {\color{black}$w=\widetilde{\ell}|_x$}};
\node at (\wi+\ti+\centerx,\centery) {\small {\color{black!50!black}$\Hh$}};
\node at (\wi+\ti-\centerx-0.5,\centery-\radius*\leng) {\small {\color{darkgray}$\breve\Cc|_x$}};

\end{tikzpicture}

\vspace{-1em}
\caption{The total dual VMRT $\breve\Cc|_x$ at $x$ when the VMRT $\Cc_x$ is finite.}
\label{fig:three_family}
\end{figure}
\vspace{-1em}

\begin{Rmk}
\label{MM:2-14}
Let $f:X=\Bl_\Gamma V_5\to V_5$ be the blow-up of the smooth del Pezzo threefold $V_5\subseteq\PP^6$ of degree $5$ along a smooth curve $\Gamma$ of degree $5$ and genus $1$ (No. 14 in \cite[Table~2]{MM}).
We have another extremal contraction $p: X\to \PP^1$ whose general fiber is a smooth del Pezzo surface of degree $5$.\linebreak
\vspace{-1em}
\[\xymatrix @R=2pc @C=2pc {
& X \ar[ld]_{f} \ar[rd]^{p} & \\
V_5 & & \PP^1
}\]
We denote by $H_1=f^*\Oo_{\PP^3}(1)$, $H_2=p^*\Oo_{\PP^1}(1)$, and $D_1$ the exceptional divisor of $f: X\to V_5$.
Then
\[
-K_X\sim H_1+H_2=2H_1-D_1,\ \ 
H_2\sim H_1-D_1.
\]

We have the following three families $\Kk_i$ of {\color{black} unbendable} rational curves on $X$:
\begin{itemize}
\item
$\Kk_1$ contains the strict transforms $\ell_1$ of some lines $l_1$ on $V_5$ not meeting $\Gamma$, 
\item
$\Kk_2$ contains the strict transforms $\ell_2$ of some conics $l_2$ on $V_5$ meeting $\Gamma$ at two points,
\item
$\Kk_3$ contains the strict transforms $\ell_3$ of some conics $l_3$ on $V_5$ meeting $\Gamma$.
\end{itemize}
We denote by $\Cc_{i,x}$ the VMRT of $\Kk_i$ at $x\in X$, $\breve\Cc_i$ the total dual VMRT on $\PP(T_X)$ associated to $\Kk_i$, and $\widetilde{\ell_i}\subseteq\PP(T_X|_{\ell_i})\subseteq\PP(T_X)$ a minimal section of $[\ell_i]\in\Kk_i$ for $i=1,\,2,\,3$.

We can observe that $\dim\Kk_1=\dim\Kk_2=2$ and $\dim\Kk_3=3$.
Thus $\Cc_{1,x}$ and $\Cc_{2,x}$ consist of a finite number of points, whereas $\Cc_{3,x}$ is a curve on $\PP(T_X|_x)$ for general $x\in X$.

We first show that $\Cc_{3,x}$ is irreducible.
It is known that the space of conics on $V_5$ is isomorphic to $\mathrm{Gr}(4,5)\cong \PP^4$ \cite[Proposition~1.2.2]{Ili}, and the subspace of conics passing through a fixed point of $V_5$ corresponds to a codimension $2$ linear subspace of $\PP^4$.
Since two general planes in $\mathbb P^4$ meet at a single point, there exists a unique conic passing through a general pair of points of $V_5$.
Moreover, as there are only finitely many conics passing through a fixed point of $V_5$ and meeting $\Gamma$ at two points with multiplicity, for general $x\in X$, we have a rational map $\Gamma \dashrightarrow \PP(\Omega_X|_x)$ sending general $y\in \Gamma$ to the tangent direction at $x$ of the conic passing through $x$ and $y\in\Gamma$, and the closure of its image is $\Cc_{3,x}$.
Thus $\Cc_{3,x}$ is irreducible.

We next show that $\Cc_{3,x}$ is not linear.
Note that $\Cc_{3,x}$ contains all the tangent directions at $x$ of conics $l\subset V_5$ with $x\in l$ and $l.\Gamma=2$.
Let $X_t$ be the fiber of $p:X\rightarrow \mathbb P^1$ with $x\in X_t$.  Then 
$\mathbb P(\Omega_{X_t}|_x)$ is a line in $\PP(\Omega_X|_x)\cong\PP^2$. 
%The strict transform of an irreducible conic $\ell$ in $V_5$ with $x\in\ell$ and $\ell.\Gamma=2$ is a rational curve $C$ in $X_t$ with $x\in C$ and $K_{X_t}^{-1}\cdot C=2$. 
Any rational curve $C$ in $X_t$ with $x\in C$ and $K_{X_t}^{-1}.C=2$ is the strict transform of some irreducible conic $l\subset V_5$ with $x\in l$ and $l.\Gamma=2$. 
Therefore, the number of the intersection points $\Cc_{3,x}\cap \mathbb P(\Omega_{X_t}|_x)$ is greater than or equal to the number of tangent directions of rational curves $C$ in $X_t$ with $x\in C$ and $K_{X_t}^{-1}.C=2$, which is greater than $1$, because we have several conic bundle structures on the del Pezzo surfaces of degree five.
Thus $\deg \Cc_{3,x}>1$.

By the previous arguments, we can now say that the total dual VMRT $\breve\Cc_3$ is a divisor on $\PP(T_X)$.
Due to the irreducibility of $\breve\Cc_{3,x}$, we have $\breve\Cc_i|_x\not\subseteq\breve\Cc_3|_x$ for general $x\in X$ since $\breve\Cc_i|_x$ is a union of a finite number of lines on $\PP(T_X|_x)$ for $i=1,\,2$.
Thus $\breve\Cc_i\not\subseteq\breve\Cc_3$, and we can conclude that $\breve\Cc_3.\widetilde{\ell_i}\geq 0$ as the minimal sections $\widetilde{\ell_i}$ cover $\breve\Cc_i$ for $i=1,\,2$.
Now, let
\[
[\breve\Cc_3]
\sim k\zeta+b_1D_1+b_2H_2
\]
for some $k>0$ and $b_1,\,b_2\in\ZZ$.
Then we have
\[
b_2=\breve\Cc_3.\widetilde{\ell_1}\geq 0
\ \ \text{and}\ \ 
b_1=\frac{1}{2}\breve\Cc_3.\widetilde{\ell_2}\geq 0.
\]
Thus by Lemma \ref{repeat}, $T_X$ is not big.
\end{Rmk}

\begin{Prop}
\label{fibration}
Let $X$ be a smooth projective threefold.
Assume that there exists a morphism $p:X\to C$ onto a smooth projective curve $C$ whose general fiber is a smooth del Pezzo surface of degree $d\leq 4$.
Then $T_X$ is not big.
\end{Prop}

\begin{proof}
We first prove the following claim.

\smallskip
\noindent \underline{Claim.}
If $S$ is a smooth del Pezzo surface of degree $d\leq 4$, then $H^0(S,\Sym^m T_S(K_S))=0$ for all $m>0$.
\smallskip

\vspace{-1.5em} %%%%%%%%%%%%%%%%%%%%%%%%%%%%%%%%%%%%%%%%%%%%%%%%%%
Assume that $d=4$.
Suppose that $H^0(S,\Sym^m T_S(K_S))\neq 0$ for some $m>0$.
Then we can take the smallest $m>0$ with $H^0(S,\Sym^m T_S(K_S))\neq 0$.
Let $D$ be an effective divisor on $\PP(T_S)$ such that $D\sim m\xi+\Lambda^*K_S$ for the tautological class $\xi=\Oo_{\PP(T_S)}(1)$ of $\Lambda:\PP(T_S)\to S$.
Following the proof of \cite[Proposition~3.5]{HLS}, there exist $5$ pairs $|\ell_{i1}|$ and $|\ell_{i2}|$ of pencils of conics on $S$ such that
%\vspace{-0.3em} %%%%%%%%%%%%%%%%%%%%%%%%%%%%%%%%%%%%%%%%%%%%%%%%%%
\[
[\ell_{i1}]+[\ell_{i2}]\sim -K_S
\ \ \text{and}\ \ 
[\breve\Cc_{i1}]+[\breve\Cc_{i2}]\sim 2\xi
\]
where $\breve\Cc_{i1}$ and $\breve\Cc_{i2}$ are the total dual VMRTs on $\PP(T_S)$ which are respectively associated to the pencils $|\ell_{i1}|$ and $|\ell_{i2}|$ for $i=1,\,\ldots,\,5$.
Then, for a minimal section $\widetilde{\ell_{ij}}\subseteq\PP(T_X|_{\ell_{ij}})\subseteq\PP(T_X)$ of $\ell_{ij}$, we have 
\[
D.\widetilde{\ell_{ij}}
=(m\xi+\Lambda^*K_S).\widetilde{\ell_{ij}}
=K_S.\ell_{ij}
<0.
\]
So $\breve\Cc_{ij}$ are contained in the support of $D$ for all $i=1,\,\ldots,\,5$ and $j=1,\,2$ as $\widetilde{\ell_{ij}}$ cover $\breve\Cc_{ij}$.
Thus there is an effective divisor
%\vspace{-0.5em}
\[
(m-10)\xi+\Lambda^*K_S
=D-\sum_{i=1}^5(\breve\Cc_{i1}+\breve\Cc_{i2})
\]
on $\PP(T_S)$, and it contradicts the minimality of $m>0$.
This proves the claim for $d=4$.

Assume that $d\leq 3$.
Then, by \cite[Theorem 1.2]{HLS}, $T_S$ is not pseudoeffective, so $H^0(S,\Sym^m T_S)=0$ for all $m>0$, and hence $H^0(S,\Sym^m T_S(K_S))=0$ for all $m>0$ because $-K_S$ is effective.
This proves the claim for $d\leq 3$.
\smallskip

Let $S=X_t$ be a general fiber of $p:X\to C$, which is a smooth del Pezzo surface of degree $d\leq 4$.
By twisting $\Oo_S(K_S)$ after taking the symmetric power to the exact sequence
%\vspace{-0.2em} %%%%%%%%%%%%%%%%%%%%%%%%%%%%%%%%%%%%%%%%%%%%%%%%%%
\[
0
\to T_S
\to T_X|_S
\to \Oo_S
\to 0,
%\vspace{-0.2em} %%%%%%%%%%%%%%%%%%%%%%%%%%%%%%%%%%%%%%%%%%%%%%%%%%
\]
we obtain the following exact sequence on $S$.
%\vspace{-0.2em} %%%%%%%%%%%%%%%%%%%%%%%%%%%%%%%%%%%%%%%%%%%%%%%%%%
\[
0
\to \Sym^m T_S(K_S)
\to \Sym^m T_X|_S(K_S)
\to \Sym^{m-1} T_X|_S(K_S)
\to 0
%\vspace{-0.2em} %%%%%%%%%%%%%%%%%%%%%%%%%%%%%%%%%%%%%%%%%%%%%%%%%%
\]
Note that $H^0(S,K_S)=0$.
Then, by induction on $m>0$, it follows that $H^0(S,\Sym^{m}T_X|_S(K_S))=0$ for all $m>0$ due to the claim.
Thus $T_X|_S$ is not big.

Suppose that $T_X$ is big.
Consider the family $\{\PP(T_X|_{X_t})\}_{t\in C}$ of codimension $1$ subvarieties of $\PP(T_X)$ which cover $\PP(T_X)$, i.\,e. $\PP(T_X)=\bigcup_{t\in C}\PP(T_X|_{X_t})$.
Notice that $\xi=\zeta|_{X_t}$ is the tautological class of $\PP(T_X|_{X_t})$.
Since $\zeta$ is big, $|m\zeta|$ gives a birational map $\PP(T_X)\dashrightarrow\PP^N$ for $m\gg 1$, so $|m\xi|$ induces a birational map $\PP(T_X|_{X_t})\dashrightarrow\PP^M$, and hence $T_X|_{X_t}$ is big for general $t\in C$, a contradiction.
Therefore, $T_X$ is not big.
\end{proof}

\begin{Rmk}
\label{del_Pezzo_fibration}
Let $f:X=\Bl_\Gamma Q\to Q$ be the blow-up of a smooth quadric threefold $Q\subseteq\PP^4$ along a smooth curve $\Gamma$ of degree $8$ and genus $5$ (No. 7 in \cite[Table~2]{MM}).
Then the other extremal contraction of $X$ is given by a fibration $p: X\to \PP^1$ whose general fiber is a smooth del Pezzo surface of degree $4$.
By Proposition~\ref{fibration}, $T_X$ is not big.
\end{Rmk}

\section{Strict Transforms of Total Dual VMRTs}

In this section, we present some criteria to determine the bigness of $T_X$ in the case when $X$ is an imprimitive Fano threefold, i.\,e., $X$ is isomorphic to the blow-up $f: X=\Bl_\Gamma Z\to Z$ of a Fano threefold $Z$ along a smooth curve $\Gamma$.
We first observe a relation between the bigness of $T_X$ and $T_Z$, and then investigate some cases when $Z$ is $\PP^3$, a smooth quadric threefold $Q\subseteq \PP^4$, or the smooth del Pezzo threefold $V_5\subseteq\PP^6$ of degree $5$.

\medskip

Let $\Phi: \PP(T_Z)\to Z$ and $\widetilde{\Phi}: \PP(f^*T_Z)\to X$ be the natural projections.
We denote by $\eta=\Oo_{\PP(T_Z)}(1)$ and $\widetilde\eta=\Oo_{\PP(f^*T_Z)}(1)$.
Then $\widetilde{\eta}={\widetilde{f}}^*\eta$ for the natural morphism $\widetilde{f}:\PP(f^*T_Z)\to\PP(T_Z)$ induced by the blow-up $f:X\to Z$.

Note that we have an exact sequence
%\vspace{-0.3em} %%%%%%%%%%%%%%%%%%%%%%%%%%%%%%%%%%%%%%%%%%%%%%%%%%
\[
0\to T_X\to f^*T_Z\to \iota_*T_{D/\Gamma}(D)\to 0
%\vspace{-0.3em} %%%%%%%%%%%%%%%%%%%%%%%%%%%%%%%%%%%%%%%%%%%%%%%%%%
\]
on $X$ where $T_{D/\Gamma}(D)$ is a locally free sheaf on the exceptional divisor $D$ with the embedding $\iota:D\hookrightarrow X$.

Maruyama's description of the elementary transformation \cite{Mar} yields the exact sequence
\[
0
\to \widetilde{\eta}\otimes \Ii_S
\to \widetilde{\eta}
\to \Oo_S(\widetilde{\eta})
\to 0
\]
on $\PP(f^*T_Z)$ where
\[
{\widetilde{\Phi}}_*(\widetilde{\eta}\otimes\Ii_S)=T_X,\quad
{\widetilde{\Phi}}_*\widetilde{\eta}=f^*T_Z,
\]
and $S\cong \PP_D(T_{D/\Gamma}(D))$ is the subvariety of $\PP(f^*T_Z)$ defined by the quotient $f^*T_Z\to \iota_*T_{D/\Gamma}(D)$.
Let $\beta$ be the blow-up of $\PP(f^*T_Z)$ along $S$.
Then there exists the commutative diagram
\begin{equation}
\label{diagram:blow-up} 
\begin{gathered}
\xymatrix@R=1.5pc@C=1pc{
& & & \Bl_S\PP(f^*T_Z) \ar[dl]_\beta \ar[dr]^\alpha &\\ 
\PP(T_Z) \ar[d]_\Phi & & \PP(f^*T_Z) \ar[d]_{\widetilde{\Phi}} \ar@{-->}[rr] \ar[ll]_{\widetilde{f}} & & \PP(T_X) \ar[d]^\Pi\\
Z & & X \ar@{=}[rr] \ar[ll]_f & & X
}\end{gathered}
\end{equation}
where $\alpha$ is the blow-down of $\Bl_S\PP(f^*T_Z)$ along the strict transform of $\widetilde{\Phi}^*D$.
So, if we denote by $R$ the exceptional divisor of $\beta$, then
\[
\beta^*\widetilde{\eta}-R \sim \alpha^*\zeta
\]
on $\Bl_S\PP(f^*T_Z)$.
Therefore, we have the following linear equivalence.
\[
\alpha_*\beta^*\widetilde{\eta}\sim \zeta+\Pi^*D
\]

\begin{Lem}[{cf. \cite[Corollary~2.3]{HLS}}]
\label{pullbackbig}
Let $E$ be a vector bundle on $X$.
Then $E$ is pseudoeffective if and only if $E\otimes \Oo_X(H)$ is big for every big $\QQ$-divisor $H$ on $X$.
\end{Lem}

\begin{proof}
Let $\zeta=\Oo_{\PP(E)}(1)$ be the tautological class of $\PP(E)$.
Assume that $\zeta+\Pi^*H$ is big for every big $\QQ$-divisor $H$ on $X$.
If $H$ is a big $\QQ$-divisor on $X$, then $\varepsilon H$ is big for any $\varepsilon>0$.
So we can observe that $\zeta$ is given by the limit of a family of big $\QQ$-divisors $\zeta+\varepsilon\Pi^*H$ as $\varepsilon\to 0$.
Thus $\zeta$ is pseudoeffective.

Conversely, assume that $\zeta$ is pseudoeffective, and let $H$ be an arbitrary big $\QQ$-divisor on $X$.
Then $H\sim A+N$ for some ample divisor $A$ and effective divisor $N$.
Because $\zeta+m\Pi^*A$ is ample on $\PP(T_X)$ for sufficiently large $m\gg 0$, we can write
\[
m(\zeta+\Pi^*H)
=(\zeta+m\Pi^*A) + m\Pi^*N + (m-1)\zeta
\]
as the sum of a big divisor $(\zeta+m\Pi^*A)+m\Pi^*N$ and a pseudoeffective divisor $(m-1)\zeta$ on $\PP(T_X)$.\linebreak
Thus $\zeta+\Pi^*H$ is big.
\end{proof}

\begin{Lem}[{cf. \cite[Corollary~2.4]{HLS}}]
\label{bigness_of_blow-up}
Let $f: X\to Z$ be the blow-up of a smooth projective variety $Z$ along a smooth subvariety.
If $T_X$ is big, then $T_Z$ is big.
Moreover, if $T_X$ is pseudoeffective, then $T_Z$ is pseudoeffective.
\end{Lem}

\begin{proof}
We continue to use the notation above.
Note that $\alpha$, $\beta$, and $\widetilde{f}$ are birational morphisms in \eqref{diagram:blow-up}.\linebreak
Thus we have the following implication.
\[
\text{$\zeta$ is big}
\ \ \Leftrightarrow\ \
\text{$\alpha^*\zeta$ is big}
\ \ \Rightarrow\ \ 
\text{$\alpha^*\zeta+R \sim (\widetilde{f}\circ \beta)^*\eta$ is big}
\ \ \Leftrightarrow\ \ 
\text{$\eta$ is big}
\]

Now, assume that $T_X$ is pseudoeffective and let $H'$ be a big $\QQ$-divisor on $Z$.
By Lemma~\ref{pullbackbig}, $T_X\otimes\Oo_X(f^*H')$ is big as $f^*H'$ is big on $X$.
Then we can show that $T_Z\otimes \Oo_Z(H')$ is big on $Z$ using the same argument as before.
Since $H'$ was arbitrary, by Lemma~\ref{pullbackbig} again, $T_Z$ is pseudoeffective.
\end{proof}

\begin{Rmk}
\label{blow-ups_of_del_pezzo_manifold}
Let $f: X=\Bl_{\Gamma}V_i\to V_i$ be the blow-up of a smooth del Pezzo threefold $V_i$ of degree $i$ for $i=1,\,2,\,3,\,4$.
Due to \cite[Theorem~1.5]{HLS}, $T_{V_i}$ is not big.
Thus by Lemma \ref{bigness_of_blow-up}, $T_X$ is not big. 
There\linebreak are $7$ deformation types of such Fano threefolds: No. 1, 3, 5, 10, 11, 16, and 19 in \cite[Table~2]{MM}.
\end{Rmk}

\begin{Rmk}
\label{MM:2-26_and_31}
Let $f: X=\Bl_\Gamma V_5\to V_5$ be the blow-up of the smooth del Pezzo threefold $V_5$ of degree~$5$ along a line $\Gamma$ (No. 26 in \cite[Table~2]{MM}).
Then the extremal contractions of $X$ are given by the blow-ups $f_1: X\to Z_1$ and $f_2: X\to Z_2$ along smooth curves where $Z_1=Q$ is a smooth quadric threefold and $Z_2=V_5$.
We denote by $H_i={f_i}^*\Oo_{Z_i}(1)$, and $D_i$ the exceptional divisor of $f_i: X\to Z_i$.
Then, from \cite[Theorem~5.1]{MM2},
\vspace{-0.3em} %%%%%%%%%%%%%%%%%%%%%%%%%%%%%%%%%%%%%%%%%%%%%%%%%%
\[
K_X\sim -H_1-H_2=-3H_1+D_1=-2H_2+D_2,\ \ 
D_1\sim 2H_1-H_2,
\ \ 
D_2\sim -H_1+H_2.
\vspace{-0.3em} %%%%%%%%%%%%%%%%%%%%%%%%%%%%%%%%%%%%%%%%%%%%%%%%%%
\]

Let $\alpha_i$, $\beta_i$, $\widetilde{f_i}$, $\Phi_i$, $\eta_i$ be defined as before.
As $\Dd_1=2\eta_1-2\Phi_1^*\Oo_{Q}(1)$ on $\PP(T_Q)$ and $\Dd_2=3\eta_2-\Phi_2^*\Oo_{V_5}(1)$ on $\PP(T_{V_5})$ are effective (for $\Dd_1$, see \cite[Theorem~5.8]{Sha} or Proposition~\ref{quadric}, and for $\Dd_2$, see \cite[Theorem~5.4]{HLS}), $\Ee_1={\alpha_1}_*\beta_1^*\widetilde{f_1}^*\Dd_1$
and $\Ee_2={\alpha_2}_*\beta_2^*\widetilde{f_2}^*\Dd_2$ are effective divisors on $\PP(T_X)$, whose linear classes are given by
\begin{align*}
[\Ee_1]
&\sim 2\zeta-2\Pi^*H_1+2\Pi^*D_1
=2\zeta+2\Pi^*H_1-2\Pi^*H_2,
\\
[\Ee_2]
&\sim 3\zeta-\Pi^*H_2+3\Pi^*D_2
=3\zeta-3\Pi^*H_1+2\Pi^*H_2.
\end{align*}
So we can write a multiple of $\zeta$ by a positive linear combination of effective divisors on $\PP(T_X)$ as
\[
5\zeta=[\Ee_1]+[\Ee_2]+\Pi^*H_1.
\]
Thus $\zeta$ is big on $\PP(T_X)$.
That is, $T_X$ is big.

Let $f: X=\Bl_\Gamma Q\to Q$ be the blow-up of a smooth quadric threefold $Q\subseteq\PP^4$ along a line $\Gamma$ (No. 31 in \cite[Table~2]{MM}).
We denote by $H=f^*\Oo_Q(1)$, and $D$ the exceptional divisor of $f: X\to Q$.
By \cite{SW}, $X$ is isomorphic to the ruled variety $\pi: X=\PP(E(1))\to\PP^2$ for the vector bundle $E$ of rank $2$ which fits into the exact sequence
\vspace{-0.3em} %%%%%%%%%%%%%%%%%%%%%%%%%%%%%%%%%%%%%%%%%%%%%%%%%%
\[
0
\to \Oo_{\PP^2}
\to E
\to \Ii_x
\to 0
%\vspace{-0.2em} %%%%%%%%%%%%%%%%%%%%%%%%%%%%%%%%%%%%%%%%%%%%%%%%%%
\]
for the ideal sheaf $\Ii_x$ of a point $x$ on $\PP^2$.
Then, for $h=\pi^*\Oo_{\PP^2}(1)$,
\vspace{-0.3em} %%%%%%%%%%%%%%%%%%%%%%%%%%%%%%%%%%%%%%%%%%%%%%%%%%
\[
-K_X\sim 3H-D=2H+h,\ \ 
h\sim H-D,\ \ 
D\sim H-h.
\vspace{-0.3em} %%%%%%%%%%%%%%%%%%%%%%%%%%%%%%%%%%%%%%%%%%%%%%%%%%
\]

Let $\breve\Cc$ be the total dual VMRT associated to the family of fibers of $\pi:X\to \PP^2$ and $\Ee=\alpha_*\beta^*{\widetilde{f}}^*\Dd$ for $\Dd=2\eta-2\Phi^*\Oo_Q(1)$ on $\PP(T_Q)$ as in the previous case.
Then, by Proposition~\ref{conic_bundle_formula}, we have
%\vspace{-0.2em} %%%%%%%%%%%%%%%%%%%%%%%%%%%%%%%%%%%%%%%%%%%%%%%%%%
\begin{align*}
[\breve\Cc]
&\sim\zeta+\Pi^*(K_X-\pi^*K_{\PP^2})
=\zeta-2\Pi^*H+2\Pi^*h,\\
[\Ee]
&\sim2\zeta-2\Pi^*H+2\Pi^*D
=2\zeta-2\Pi^*h.
%\vspace{-0.2em} %%%%%%%%%%%%%%%%%%%%%%%%%%%%%%%%%%%%%%%%%%%%%%%%%%
\end{align*}
So we can write a multiple of $\zeta$ by a positive linear combination of effective divisors on $\PP(T_X)$ as
\vspace{-0.2em} %%%%%%%%%%%%%%%%%%%%%%%%%%%%%%%%%%%%%%%%%%%%%%%%%%
\[
3\zeta
=[\breve\Cc]+[\Ee]+2\Pi^*H.
%\vspace{-0.5em} %%%%%%%%%%%%%%%%%%%%%%%%%%%%%%%%%%%%%%%%%%%%%%%%%%
\]
Thus $\zeta$ is big on $\PP(T_X)$.
That is, $T_X$ is big.
\end{Rmk}

\begin{Rmk}
\label{divisor_from_projective_space}
Let $f: X=\Bl_{\Gamma}\PP^3\to \PP^3$ be the blow-up of the projective space $\PP^3$ along a smooth curve $\Gamma$.
We denote by $H=f^*\Oo_{\PP^3}(1)$, and $D$ the exceptional divisor of $f: X\to \PP^3$.
From the exact sequence
%\vspace{-0.3em} %%%%%%%%%%%%%%%%%%%%%%%%%%%%%%%%%%%%%%%%%%%%%%%%%%
\[
0
\to \Oo_{\PP(f^*T_{\PP^3})}(k\widetilde{\eta}-k\Pi^*H)\otimes\Ii_S
\to \Oo_{\PP(f^*T_{\PP^3})}(k\widetilde{\eta}-k\Pi^*H)
\to \Oo_S(k\widetilde{\eta}-k\Pi^*H)
\to 0
%\vspace{-0.2em} %%%%%%%%%%%%%%%%%%%%%%%%%%%%%%%%%%%%%%%%%%%%%%%%%%
\]
on $\PP(f^*T_{\PP^3})$, we can observe that there exists an effective divisor $\Ee_0\sim k\widetilde{\eta}-kH$ on $\PP(f^*T_{\PP^3})$ which vanishes on $S$ for some $k>0$.
Indeed, $h^0(S^k(T_{\PP^3}(-1)))=\Theta(k^3)$ is observed from the exact sequence
%\vspace{-0.2em} %%%%%%%%%%%%%%%%%%%%%%%%%%%%%%%%%%%%%%%%%%%%%%%%%%
\[
0
\to {\Oo_{\PP^3}(-1)}^{\oplus \genfrac(){0pt}{}{k+2}{3}}
\to {\Oo_{\PP^3}}^{\oplus \genfrac(){0pt}{}{k+3}{3}} \to
S^k(T_{\PP^3}(-1))
\to 0
%\vspace{-0.3em} %%%%%%%%%%%%%%%%%%%%%%%%%%%%%%%%%%%%%%%%%%%%%%%%%%
\]
obtained by taking symmetric powers to the Euler sequence, whereas $h^0(\Oo_S(k\widetilde{\eta}-k\Pi^*H))=O(k^2)$ as $\Oo_S(k\widetilde{\eta}-k\Pi^*H)$ is a line bundle on the surface $S\cong\PP_D(T_{D/\Gamma}(D))$.
Thus we obtain an effective divisor
\[
[\Ee]
\sim k\alpha_*\beta^*\widetilde{\eta}-k\Pi^*H-m\Pi^*D
=k\zeta-k\Pi^*H+(k-m)\Pi^*D
%\vspace{-0.2em} %%%%%%%%%%%%%%%%%%%%%%%%%%%%%%%%%%%%%%%%%%%%%%%%%%
\]
on $\PP(T_X)$ for some $m>0$, which is the strict transform of $\Ee_0$ from $\PP(f^*T_{\PP^3})$ to $\PP(T_X)$.

If $\Gamma$ is degenerate (No. 28 and 30 in \cite[Table~2]{MM} for instance), then $H-D$ is effective on $X$.
So we can write a multiple of $\zeta$ by a positive linear combination of effective divisors on $\PP(T_X)$ as
%\vspace{-0.2em} %%%%%%%%%%%%%%%%%%%%%%%%%%%%%%%%%%%%%%%%%%%%%%%%%%
\[
k\zeta
=[\Ee]+k\Pi^*(H-D)+m\Pi^*D
%\vspace{-0.2em} %%%%%%%%%%%%%%%%%%%%%%%%%%%%%%%%%%%%%%%%%%%%%%%%%%
\]
for some $k>0$ and $m>0$.
Thus $\zeta$ is big on $\PP(T_X)$.
That is, $T_X$ is big.

If $\Gamma$ is a twisted cubic curve (No. 27 in \cite[Table~2]{MM}), then, by \cite{SW}, $X$ is isomorphic to the\linebreak ruled variety $\pi: X=\PP(E(1))\to\PP^2$ for the vector bundle $E$ of rank $2$ which fits into the exact sequence
\[
0
\to \Oo_{\PP^2}(-1)^{\oplus 2}
\to {\Oo_{\PP^2}}^{\oplus 4}
\to E(1)
\to 0
%\vspace{-0.3em} %%%%%%%%%%%%%%%%%%%%%%%%%%%%%%%%%%%%%%%%%%%%%%%%%%
\]
on $\PP^2$.
Then, for $h=\pi^*\Oo_{\PP^2}(1)$, 
%\vspace{-0.2em} %%%%%%%%%%%%%%%%%%%%%%%%%%%%%%%%%%%%%%%%%%%%%%%%%%
\[
-K_X\sim 4H-D=2H+h,\ \ 
h\sim 2H-D,\ \
D\sim 2H-h.
%\vspace{-0.2em} %%%%%%%%%%%%%%%%%%%%%%%%%%%%%%%%%%%%%%%%%%%%%%%%%%
\]

Let $\breve\Cc$ be the total dual VMRT on $\PP(T_X)$ associated to the family of fibers of $\pi:X\to \PP^2$.
Then, by Proposition~\ref{conic_bundle_formula}, we have
%\vspace{-0.2em} %%%%%%%%%%%%%%%%%%%%%%%%%%%%%%%%%%%%%%%%%%%%%%%%%%
\[
[\breve\Cc]
\sim\zeta+\Pi^*(K_X-\pi^*K_{\PP^2})
=\zeta-2\Pi^*H+2\Pi^*h
=\zeta+2\Pi^*H-2\Pi^*D.
%\vspace{-0.2em} %%%%%%%%%%%%%%%%%%%%%%%%%%%%%%%%%%%%%%%%%%%%%%%%%%
\]
So we can write a multiple of $\zeta$ by a positive linear combination of effective divisors on $\PP(T_X)$ as
%\vspace{-0.2em} %%%%%%%%%%%%%%%%%%%%%%%%%%%%%%%%%%%%%%%%%%%%%%%%%%
\[
3k\zeta
=k[\breve\Cc]+2[\Ee]+2m\Pi^*D
%\vspace{-0.2em} %%%%%%%%%%%%%%%%%%%%%%%%%%%%%%%%%%%%%%%%%%%%%%%%%%
\]
for some $k>0$ and $m>0$.
Thus $\zeta$ is big on $\PP(T_X)$.
That is, $T_X$ is big.
\end{Rmk}

\begin{Prop}
\label{quadric}
Let $f:X=\Bl_\Gamma Q\to Q$ be the blow-up of a smooth quadric threefold $Q\subseteq\PP^4$ along a smooth curve $\Gamma$.
Then there exists an irreducible and reduced effective divisor
\[
[\breve\Cc]
\sim 2\zeta-2\Pi^*H+2\Pi^*D
\]
on $\PP(T_X)$ which satisfies \eqref{condition} for $H=f^*\Oo_{Q}(1)$ and the exceptional divisor $D$ of $f:X\to Q$.
\end{Prop}

\begin{proof}
Let $\Mm$ be the family of lines $\elll$ on $Q$ which is known to be $\Mm\cong\PP^3$.
Then $N_{\elll|Q}\cong\Oo_{\PP^1}(1)\oplus\Oo_{\PP^1}$, and $\Mm$ is a family of {\color{black} unbendable} rational curves on $Q$.
Let $\breve\Dd$ be the total dual VMRT on $\PP(T_Q)$ associated to $\Mm$.
As $Q$ is homogeneous and the VMRT $\Dd_{z}$ of $\Mm$ at $z\in Q$ is irreducible and reduced, so is $\breve\Dd|_z$ for all $z\in Q$, and hence $\breve\Dd$ is irreducible and reduced.

Let $H'=\Oo_Q(1)$ and $\eta=\Oo_{\PP(T_Q)}(1)$ be the tautological class of $\Phi: \PP(T_Q)\to Q$.
It is known that 
\[
[\breve\Dd]\sim 2\eta-2\Phi^*H'
\]
(see \cite[Proof of Proposition~3.1]{SCW} for instance). Therefore we can write
\[
[\widetilde{f}^*\breve\Dd]=2\widetilde{f}^*\eta-2\widetilde{\Phi}^*H
\]
where $\widetilde{f}: \PP(f^*T_Q)\to \PP(T_Q)$ and $\widetilde{\Phi}:\mathbb P(f^*T_Q)\rightarrow X$ are the natural morphisms induced by  $f: X\to Q$ and $\Phi: \PP(T_Q)\to Q$.

Let $\Kk$ be the family of rational curves on $X$ containing the strict transform $[\ell]$ of $[\elll]$ for general $[\elll]\in\Mm$.
Since a general member $[\elll]\in\Mm$ does not intersect with the blow-up center $\Gamma$, $\Kk$ is a family of {\color{black} unbendable} rational curves on $X$.

Let $\breve\Cc$ be the total dual VMRT  on $\PP(T_X)$ associated to $\Kk$.
Because the total dual VMRT $\breve\Dd$ is given by the union of minimal sections of rational curves of $\Mm$, $\breve\Cc|_U$ and $\breve\Dd|_V$ coincide under the isomorphism $\PP(T_X|_{U})\cong \PP(T_Q|_{Z})$ over $U=X\backslash D$ and $V=Q\backslash\Gamma$.
As the closure is uniquely determined by a subset out of codimension $1$, $\breve\Cc$ is the strict transform of $\breve\Dd$ from $\PP(T_Q)$ to $\PP(T_X)$.
Moreover, $\breve\Cc$ is the strict transform of $\widetilde{f}^*\breve\Dd$ from $\PP(f^*T_Q)$ to $\PP(T_X)$.
As $\breve\Dd$ is irreducible and reduced, so is its strict transform~$\breve\Cc$.

Let $m\geq 0$ be the order of vanishing of $\widetilde{f}^*\breve\Dd$ on $S\cong\PP_{D}(T_{D/\Gamma}(D))\subseteq\PP(f^*T_Q)$.
Then the strict transform $\breve\Cc$ of $\widetilde{f}^*\breve\Dd$ from $\PP(f^*T_Q)$ to $\PP(T_X)$ satisfies
\[
[\breve\Cc]
\sim \alpha_*\beta^*(2\widetilde{\eta}-2\widetilde{\Phi}^*H)-m\Pi^*D
=2\zeta-2\Pi^*H+(2-m)\Pi^*D.
\]
We will complete the proof by showing that $m=0$, which is equivalent to showing that $S\not\subseteq \widetilde{f}^*\breve\Dd$.

Let $z\in \Gamma$.
Then $\PP(f^*T_Q|_{f^{-1}(z)})\cong \PP(T_Q|_z)\times f^{-1}(z)$.
Moreover, $\PP(T_{D/\Gamma}(D)|_{f^{-1}(z)})$ is a subvariety of $\PP(f^*T_Q|_{f^{-1}(z)})$ whose image is $\widetilde{f}(\PP(T_{D/\Gamma}(D)|_{f^{-1}(z)}))=\PP(N_{\Gamma|Q}|_z)$ in $\PP(T_Q|_z)$.
Since the line $\PP(N_{\Gamma|Q}|_z)$ is not contained in the conic $\breve\Dd|_z\subseteq\PP(T_Q|_z)$, ${\widetilde{f}}^*\breve\Dd|_z$ does not contain  $\PP(T_{D/\Gamma}(D)|_{f^{-1}(z)})=S|_{f^{-1}(z)}$ (see Figure \ref{fig:strict_transform}).
Thus we can conclude that $S\not\subseteq \widetilde{f}^*\breve\Dd$.
\end{proof}

\begin{figure}

\begin{tikzpicture}[x=(0:1.0), y=(90:0.3), z=(52:0.3),
                    line cap=round, line join=round,
                    axis/.style={black, thick,->},
                    vector/.style={>=stealth,->}]

\pgfmathsetmacro{\op}{0.5}

\pgfmathsetmacro{\wi}{10}
\pgfmathsetmacro{\he}{10}
\pgfmathsetmacro{\le}{5}
\pgfmathsetmacro{\ti}{2}
\pgfmathsetmacro{\siw}{1}
\pgfmathsetmacro{\sih}{3}
\pgfmathsetmacro{\A}{0.15}
\pgfmathsetmacro{\AA}{0.16}
\pgfmathsetmacro{\B}{0.5}
\pgfmathsetmacro{\C}{0.83}
\pgfmathsetmacro{\D}{0.87}
\pgfmathsetmacro{\DD}{0.88}
\pgfmathsetmacro{\E}{0.95}

%Q
\draw[] (-\ti-\siw,-\sih,0)--(-\ti+\siw,-\sih,0)--(-\ti+\siw,-\sih,\wi)--(-\ti-\siw,-\sih,\wi)--cycle;
%\Gamma
\draw[thick] (-\ti,-\sih,0)--(-\ti,-\sih,10);
%from p to T(Q|_p)
\draw[draw=none, fill=gray!50!white, fill opacity=0.3] (-\ti,-\sih,5)--(-\ti,0,0)--(-\ti,0,10)--cycle;

%\breve\Cc_p
\draw[very thick, color=darkgray] (-\ti,1,0) to[out=30,in=200] (-\ti,8,6);
\draw[very thick, color=darkgray] (-\ti,8,6) to[out=10,in=150] (-\ti,3,10);
%PP(N|_p)
\draw[very thick, color=black!50!black] (-\ti,6,0)--(-\ti,6,10);

%T_Q
\draw[color=gray] (-\ti,0,0)--(-\ti,\he,0)--(-\ti,\he,\wi)--(-\ti,0,\wi)--cycle;

%X
\draw[] (+\ti-\siw,-\sih,0)--(+\ti+\le+\siw,-\sih,0)--(+\ti+\le+\siw,-\sih,\wi)--(+\ti-\siw,-\sih,\wi)--cycle;
%D
\draw[draw=none, fill=black, opacity=0.1] (+\ti,-\sih,0)--(+\ti+\le,-\sih,0)--(+\ti+\le,-\sih,\wi)--(+\ti,-\sih,\wi)--cycle;

%from f^{-1}(p) to T(T_X|_f^{-1}(p))
\draw[draw=none, fill=gray!50!white, fill opacity=0.3] (+\ti+\le*\B,-\sih,5)--(+\ti+\le*\B,0,0)--(+\ti+\le*\B,0,10)--cycle;

\draw[color=gray, opacity=0.5, densely dashed] (+\ti,-\sih,5)--(+\ti,0,0);
\draw[color=gray, opacity=0.5, densely dashed] (+\ti,-\sih,5)--(+\ti,0,10);
\draw[color=gray, opacity=0.5, densely dashed] (+\ti+\le,-\sih,5)--(+\ti+\le,0,0);
\draw[color=gray, opacity=0.5, densely dashed] (+\ti+\le,-\sih,5)--(+\ti+\le,0,10);

%f^{-1}(p)
\draw[thick] (+\ti,-\sih,5)--(+\ti+\le,-\sih,5);

%horizontal masking
\draw[draw=none, fill=white, fill opacity=0.5] (+\ti,0,0)--(+\ti+\le,0,0)--(+\ti+\le,0,\wi)--(+\ti,0,\wi)--cycle;

%\widetilde{f}^*\breve\Dd left
\draw[draw=none, fill=darkgray!20!white,opacity=0.2]
    (\ti,1,0) to[out=30,in=200] (\ti,8,6) -- (\ti+\le*\B,8,6) to[out=200,in=30] (\ti+\le*\B,1,0) -- cycle;
\draw[draw=none, fill=darkgray!20!white,opacity=0.2]
    (\ti,8,6) to[out=10,in=150] (\ti,3,10) -- (\ti+\le*\B,3,10) to[out=150,in=10] (\ti+\le*\B,8,6) -- cycle;

%S left
\draw[very thick, color=black!50!black, opacity=0.7] (\ti,6,0)--(\ti+\le*\A,6,\wi*\A);
\draw[thick, color=black!50!black, densely dashed, opacity=0.7] (\ti+\le*\A,6,\wi*\A)--(\ti+\le*\B,6,\wi*\B);

%vertical masking
\draw[color=gray, fill=white, fill opacity=0.7] (+\ti+\le*\B,0,0)--(+\ti+\le*\B,0,\wi)--(+\ti+\le*\B,\he,\wi)--(+\ti+\le*\B,\he,0)--cycle;

%\widetilde{f}^*\breve\Dd right
\draw[draw=none, fill=darkgray!20!white,opacity=0.2]
    (\ti+\le*\B,1,0) to[out=30,in=200] (\ti+\le*\B,8,6) -- (\ti+\le,8,6) to[out=200,in=30] (\ti+\le,1,0) -- cycle;
\draw[draw=none, fill=darkgray!20!white,opacity=0.2]
    (\ti+\le*\B,8,6) to[out=10,in=150] (\ti+\le*\B,3,10) -- (\ti+\le,3,10) to[out=150,in=10] (\ti+\le,8,6) -- cycle;
    
%\widetilde{f}^*\breve\Dd|_x
\draw[very thick, color=darkgray] (+\ti+\le*\B,1,0) to[out=30,in=200] (+\ti+\le*\B,8,6);
\draw[very thick, color=darkgray] (+\ti+\le*\B,8,6) to[out=10,in=150] (+\ti+\le*\B,3,10);

%S right
\draw[thick, color=black!50!black, densely dashed, opacity=0.7] (\ti+\le*\B,6,\wi*\B)--(\ti+\le*\C,6,\wi*\C);
\draw[very thick, color=black!50!black, opacity=0.7] (\ti+\le*\C,6,\wi*\C)--(\ti+\le*\D,6,\wi*\D);
\draw[thick, color=black!50!black, densely dashed, opacity=0.7] (\ti+\le*\D,6,\wi*\D)--(\ti+\le*\E,6,\wi*\E);
\draw[very thick, color=black!50!black, opacity=0.7] (\ti+\le*\E,6,\wi*\E)--(\ti+\le,6,\wi);

%T(f^*T_Q)
\draw[color=gray, densely dashed, opacity=0.5] (+\ti,0,0)--(+\ti,\he,0)--(+\ti,\he,\wi)--(+\ti,0,\wi)--cycle;
\draw[color=gray, densely dashed, opacity=0.5] (+\ti+\le,0,0)--(+\ti+\le,\he,0)--(+\ti+\le,\he,\wi)--(+\ti+\le,0,\wi)--cycle;
\draw[color=gray, densely dashed, opacity=0.5] (+\ti,0,0)--(+\ti+\le,0,0);
\draw[color=gray, densely dashed, opacity=0.5] (+\ti,\he,0)--(+\ti+\le,\he,0);
\draw[color=gray, densely dashed, opacity=0.5] (+\ti,0,\wi)--(+\ti+\le,0,\wi);
\draw[color=gray, densely dashed, opacity=0.5] (+\ti,\he,\wi)--(+\ti+\le,\he,\wi);

\node[circle, fill=black!50!black, inner sep=0pt, minimum size=2pt] at (\ti+\le*\AA,6,\wi*\AA) {};
\node[circle, fill=black!50!black, inner sep=0pt, minimum size=2pt] at (\ti+\le*\DD,6,\wi*\DD) {};

\node[circle, fill=black, inner sep=0pt, minimum size=3pt] at (-\ti,-\sih,\wi/2) {};
\node at (-\ti+0.5,-\sih,\wi/2-3) {\small $z$};

\node[circle,fill=black, inner sep=0pt, minimum size=3pt] at (\ti+\le*\B,-\sih,\wi/2) {};
\node at (\ti+\le*\B+0.5,-\sih,\wi/2-3) {\small $x$};

\node at (-\ti+0.6,3,1.5) {\color{darkgray} \small $\breve\Dd|_z$};
\node[circle, fill=black, inner sep=0pt, minimum size=3pt] at (-\ti,6,\wi*\B) {};
\node at (-\ti+0.1,5,5) {\color{black} \small $v$};

\node at (+\ti+\le*\B+0.75,3,1.5) {\color{darkgray} \small $\widetilde{f}^*\breve\Dd|_x$};
\node[circle, fill=black, inner sep=0pt, minimum size=3pt] at (\ti+\le*\B,6,\wi*\B) {};
\node at (+\ti+\le*\B+0.2,5,5) {\color{black} \small $w=S|_x$};

\node at (-\ti-0.8,5,0) {\color{black!50!black} \small $\PP(N_{\Gamma|Q}|_z)$};
\node at (-\ti-0.65,1,0) {\color{black} \small $\PP(T_Q|_z)$};

\node at (+\ti+\le+1.43,7,10) {\color{black!50!black} \small $\PP(T_{D/\Gamma}(D)|_{f^{-1}(z)})$};
\node at (+\ti+\le+1.2,1,10) {\color{black} \small $\PP(f^*T_Q|_{f^{-1}(z)})$};
\node at (+\ti+\le*\B-0.8,1,0) {\color{black} \small $\PP(f^*T_Q|_x)$};

\node at (-\ti-\siw,-\sih+1,0) {\small $Q$};
\node at (+\ti+\le+\siw+1.5,-\sih+1,0) {\small $X$};

\node at (-\ti,-\sih+1,\wi+0.5) {\small $\Gamma$};
\node at (+\ti+\le-0.1,-\sih+1,\wi+0.5) {\small $D$};
\node at (+\ti,-\sih+0.2,\wi/2-3) {\small $f^{-1}(z)$};
\end{tikzpicture}

\vspace{-0.5em}
\caption{The restriction $\widetilde{f}^*\breve\Dd|_x\subseteq\PP(f^*T_Q|_x)$ of $\widetilde{f}^*\breve\Dd\subseteq\PP(f^*T_Q)$ over $x\in D$.}
\label{fig:strict_transform}
%\vspace{-1.5em}
\end{figure}

\begin{Rmk}
\label{blow-ups_of_quadric_threefold}
Let $f:X=\Bl_{\Gamma}Q\to Q$ be the blow-up of a smooth quadric threefold $Q\subseteq\PP^4$ along a smooth curve $\Gamma$ of degree $4$ and genus $0$ (No. 21 in \cite[Table2]{MM}).
Then the extremal contractions of $X$ are given by two distinct blow-ups $f_i: X\to Q$ for $i=1,\,2$.
We denote by $H_i=f_i^*\Oo_Q(1)$, and $D_i$ the exceptional divisor of $f_i: X\to Q$ for $i=1,\,2$.
Then
\[
-K_X\sim H_1+H_2=3H_1-D_1=3H_2-D_2,\ \ 
D_1\sim 2H_1-H_2,\ \ 
D_2\sim -H_1+2H_2.
\]
Let $\breve\Cc_i$ be the total dual VMRT on $\PP(T_X)$ given by Proposition~\ref{quadric} for $i=1,\,2$.
Then the sum of two effective divisors
\begin{align*}
[\breve\Cc_1]
&=2\zeta-2\Pi^*H_1+2\Pi^*D_1
=2\zeta+2\Pi^*H_1-2\Pi^*H_2,
\\
[\breve\Cc_2]
&=2\zeta-2\Pi^*H_2+2\Pi^*D_2
=2\zeta-2\Pi^*H_1+2\Pi^*H_2
\end{align*}
yields an effective divisor on $\PP(T_X)$ satisfying \eqref{condition}.
Thus by Lemma \ref{repeat}, $\zeta$ is not big on $\PP(T_X)$.
\end{Rmk}

\begin{Prop}
\label{V5}
Let $f:X=\Bl_\Gamma V_5\to V_5$ be the blow-up of the smooth del Pezzo threefold $V_5$ of degree $5$ along a smooth curve $\Gamma$ contained in an intersection of two hyperplane sections, which is not a line on $V_5$.
Then there exists an irreducible and reduced effective divisor
\[
[\breve\Cc]
\sim 3\zeta-\Pi^*H+3\Pi^*D
\]
on $\PP(T_X)$ which satisfies \eqref{condition} for $H=f^*\Oo_{V_5}(1)$ and the exceptional divisor $D$ of $f:X\to V_5$.
\end{Prop}

\begin{proof}
Let $\breve\Dd$ be the total dual VMRT on $\PP(T_{V_5})$ associated to the family of lines on $V_5$.
Then, according to \cite[Theorem~5.4]{HLS},
\[
[\breve\Dd]
\sim 3\eta-\Phi^*H'.
\]
for $H'=\Oo_{V_5}(1)$ and the tautological class $\eta=\Oo_{\PP(T_{V_5})}(1)$ of $\Phi: \PP(T_{V_5})\to {V_5}$.

Let $\breve\Cc$ be the total dual VMRT on $\PP(T_X)$ associated to the family of the strict transforms of lines on $V_5$.
Then $\breve\Cc$ is the strict transform of $\breve\Dd$ from $\PP(T_{V_5})$ to $\PP(T_X)$.
By the same argument in the proof of Proposition~\ref{quadric}, the proof is completed once we show that $S\not\subseteq \widetilde{f}^*\breve\Dd$.

The VMRT of the family of lines on $V_5$ at a point $z\in V_5$ is the union of three points on $\PP(\Omega_{V_5}|_z)$ with multiplicity \cite[Lemma 2.3~(1)]{FN}.
So $\breve\Dd|_z$ is the union of three lines on $\PP(T_{V_5}|_z)$.
If there is a point $z\in \Gamma$ such that every line passing through $z$ is not tangent to $\Gamma$ at $z$, then $\PP(N_{\Gamma|V_5}|_z)\not\subseteq \breve\Dd|_z$ on $\PP(T_{V_5}|_z)$, and it implies that $S\not\subseteq \widetilde{f}^*\breve\Dd$.
Thus it suffices to show that there are only finitely many lines on $V_5$ which is tangent to $\Gamma$.

We consider $V_5$ as the subvariety of $\PP^6$ and fix a $4$-dimensional linear subspace $L$ of $\PP^6$ containing $\Gamma$.
Then $L\cap V_5$ is a union of finitely many curves on $V_5$.
If a line $l$ on $\PP^6$ is tangent to $\Gamma$, then $l\subseteq L$, and if $l$ is also contained in $V_5$, then $l\subseteq L\cap V_5$.
Therefore, we can conclude that there are only finitely many lines $l$ on $V_5$ which is tangent to $\Gamma$.
\end{proof}

\section{Total Dual VMRTs on Blow-Ups}

In this section, developing the idea of \cite[Lemma 5.3]{HLS}, we present some formulas for total dual VMRTs on $\PP(T_X)$, which are distinctly constructed from the ones in the previous section, in the case when $X$ is an imprimitive Fano threefold.

\begin{Prop}
\label{space}
Let $f:X=\Bl_\Gamma\PP^3\to\PP^3$ be the blow-up of the projective space $\PP^3$ along a smooth nondegenerate curve $\Gamma$ of degree $d$ and genus $g$.
Assume that $\Gamma$ has at most a finite number of quadrisecant lines on $\PP^3$.
Then there exists an irreducible and reduced effective divisor
{\small \[
[\breve\Cc]
\sim
\begin{cases}
\left(\frac{1}{2}(d-1)(d-2)-g\right)\zeta
+\left(d+g-1\right)\Pi^*H
-\left((d-1)-\frac{1}{2}(d-2)(d-3)+g\right)\Pi^*D
& \text{if $\Gamma$ has a trisecant}\\
\left(\frac{1}{2}(d-1)(d-2)-g\right)\zeta
+\left(d+g-1\right)\Pi^*H
-\left(d-1\right)\Pi^*D
& \text{otherwise}
\end{cases}
\]}
on $\PP(T_X)$ which satisfies \eqref{condition} for $H=f^*\Oo_{\PP^3}(1)$ and the exceptional divisor $D$ of $f:X\to \PP^3$.
\end{Prop}

\begin{proof}
Let $\Kk$ be the family of {\color{black} unbendable} rational curves on $X$ containing the strict transforms of general secant lines of $\Gamma$ on $\PP^3$.
Note that $N_{\ell|X}\cong\Oo_{\PP^1}\oplus\Oo_{\PP^1}$ for general $[\ell]\in\Kk$.
We can observe that $\overline{\Kk}\cong S^2\Gamma$.
Indeed, $S^2\Gamma$ is smooth and there exists a finite morphism from $S^2\Gamma$ to the locus of secant lines of $\Gamma$ in the Grassmannian $\text{Gr}(2,4)$ of lines on $\PP^3$.
Let $\qq_0:\overline{\Uu_0}\to\overline{\Kk}$ be the universal family of the secant lines of $\Gamma$ on $\PP^3$ and ${\ee_0}:\overline{\Uu_0}\to\PP^3$ be its evaluation morphism.
Note that $\qq_0:\overline{\Uu_0}\to\overline{\Kk}$ is a $\PP^1$-fibration, and $\overline{\Uu_0}\cong \PP_{\overline\Kk}(V)$ for some vector bundle $V$ of rank $2$ on $\overline\Kk$, where $V$ is given by the pull-back of the universal bundle via the induced morphism $\overline{\Kk}\to \mathrm{Gr}(2,4)$.

As a general hyperplane meets a general secant line at $1$ point on $\PP^3$, we take $V={\qq_0}_*{\ee_0}^*\Oo_{\PP^3}(1)$ so that $\Oo_{\PP_{\overline\Kk}(V)}(1)={\ee_0}^*\Oo_{\PP^2}(1)$. 

We first prove the case where $\Gamma$ has a trisecant line.
Let $\qq:\overline\Uu\to \overline\Kk$ be the normalization of the universal family of curves on $X$ with the evaluation morphism $\ee: \overline\Uu\to X$.
Then $\qq:\overline\Uu\to \overline\Kk$ is a conic bundle over $\overline\Kk$ whose discriminant locus is the locus of trisecant lines of $\Gamma$ on $\PP^3$.
Moreover, $\overline\Uu$ is obtained by the blow-up $\ggg:\overline{\Uu}\to\overline{\Uu_0}$ along some curve in $\overline{\Uu_0}$ over the locus in $\overline\Kk$ of trisecant lines.
\[\xymatrix @R=2pc @C=3pc {
\overline\Kk \ar@{=}[d] & \overline\Uu \ar[r]^{\ee} \ar[l]_{\qq} \ar[d]_\ggg &
X \ar[d]^{f} \\
\overline\Kk  &
\overline{\Uu_0}\cong \PP_{\overline\Kk}(V) \ar[r]_{\ee_0} \ar[l]^{\qq_0} \ar@{-->}[ur]&
\PP^3
}
\]

Let $\xi=\ggg^*\Oo_{\PP_{\overline\Kk}(V)}(1)$ and $E$ be the exceptional divisor of $\ggg:\overline{\Uu}\to\overline{\Uu_0}$.
Then we can write
\[
K_{\overline\Uu/\overline\Kk}
=\ggg^*K_{\PP_{\overline\Kk}(V)/\overline\Kk}+E
=-2\xi+\qq^*c_1(V)+E.
\]

Let $\breve\Cc$ be the total dual VMRT on $\PP(T_X)$ associated to $\Kk$.
Then $\breve\Cc$ is irreducible by its construction, and it is reduced as a general VMRT $\Cc_x$ of $\overline\Kk$ is reduced for general $x\in X$.
Moreover, as $\Kk$ is locally unsplit, we can apply Proposition~\ref{conic_bundle_formula} to obtain the linear equivalence
\[
[\breve\Cc]
\sim k\zeta+\Pi^*\ee_*(K_{\overline\Uu/\overline\Kk})
\]
on $\PP(T_X)$ for $k=\deg (\ee)$.
From the commutativity, we have $\xi=\ggg^*{\ee_0}^*\Oo_{\PP^3}(1)=\ee^*f^*\Oo_{\PP^3}(1)=\ee^*H$, and
%\vspace{-0.3em} %%%%%%%%%%%%%%%%%%%%%%%%%%%%%%%%%%%%%%%%%%%%%%%%%%
\[
\ee_*(K_{\overline{\Uu}/\overline{\Kk}})
=-2\ee_*\xi +\ee_*\qq^*c_1(V)+ \ee_*E
=-2kH+\ee_*\qq^*c_1(V)+ \ee_*E.
\]

First, let
%\vspace{-0.3em} %%%%%%%%%%%%%%%%%%%%%%%%%%%%%%%%%%%%%%%%%%%%%%%%%%
\[
\ee_*\qq^*c_1(V)=aH+bD
\]
for some $a,\,b\in\ZZ$.
We know from blow-up formulas (e.\,g. see \cite[(4.3)]{MM2}) that
\[
H^3=1,\ \ 
H^2.D=0,\ \ 
H.D^2=-d.
%\vspace{-0.5em} %%%%%%%%%%%%%%%%%%%%%%%%%%%%%%%%%%%%%%%%%%%%%%%%%%
\]
Thus
%\vspace{-0.3em} %%%%%%%%%%%%%%%%%%%%%%%%%%%%%%%%%%%%%%%%%%%%%%%%%%
\[
aH^3
=\ee_*\qq^*c_1(V).H^2
=\qq^*c_1(V).\xi^2
=\qq^*c_1(V).(\qq^*c_1(V)\xi-\qq^*c_2(V))
={c_1(V)}^2,
\]
and ${c_1(V)}^2=kH^3+c_2(V)$ due to
\begin{align*}
kH^3
=\xi^3
=\qq^*c_1(V)\xi^2-\qq^*c_2(V)\xi
=\qq^*c_1(V)(\qq^*c_1(V)\xi-\qq^*c_2(V))-\qq^*c_2(V)\xi
={c_1(V)}^2-c_2(V).
\end{align*}
Then the number $r:=c_2(V)$ can be calculated by the number of points of the degeneracy locus of a general member of $|\Oo_{\PP_{\overline\Kk}(V)}(1)|$.
That is, $r=\#\{[\ell]\in\overline\Kk\,|\,\ell\subseteq P\}$ for general $P\in|H|$.

On the other hand,
\[
bH.D^2
=\ee_*\qq^*c_1(V).H.D
=\qq^*c_1(V).\xi.\ee^*D
=(\xi^2+\qq^*c_2(V)).\ee^*D
=kH^2.D+\qq^*c_2(V).\ee^*D
=2c_2(V)
\]
as a general secant line meets $\Gamma$ at $2$ points on $\PP^3$  so that $\ee^*D\equiv 2\xi$ in $\text{N}^1(\overline\Uu/\overline\Kk)$.

Moreover, we can deduce that
\[
\ee_*E=mD
\]
for some $m>0$ from the observation that the image $\ee(E)$ lies on the exceptional locus of $f:X\to\PP^3$, and the multiplicity $m$ is given by the number of trisecant lines of $\Gamma$ on $\PP^3$ passing through a general point of $\Gamma$.

By the projection $\pi_x$ from a general point $x$ in $\mathbb P^3$, it gives a one-to-one correspondence between the number of secant lines of $\Gamma$ through $x$ and the number of nodes of the curve $\pi_x(\Gamma)$ in $\PP^2$, which is $k$.
Also, by the projection $\pi_y$ from a general point $y\in \Gamma$, it gives a one-to-one correspondence between the number of trisecant lines of $\Gamma$ through $y$ and the number of nodes of the curve $\pi_y(\Gamma)$ in $\PP^2$, which is $m$.
Then the genus-degree formula for plane curves gives $k$ and $m$.
The number $r$ can also be computed (cf. \cite[Lemma 5.3]{HLS}).
That is,
\begin{align*}
k
&=(\text{the number of nodes in a general hyperplane projection $\Gamma\to\PP^2$})
=\frac{1}{2}(d-1)(d-2)-g,
\\%[-0.2\baselineskip] %%%%%%%%%%%%%%%%%%%%%%%%%%%%%%%%%%%%%%%%%%%%%%%%%%
r
&=(\text{the number of choices of two points among $\Gamma\cap\PP^2$ for a general hyperplane $\PP^2\subseteq \PP^3$})
=\genfrac(){0pt}{}{d}{2},
\\%[-0.2\baselineskip] %%%%%%%%%%%%%%%%%%%%%%%%%%%%%%%%%%%%%%%%%%%%%%%%%%
m
&=(\text{the number of nodes in the projection $\Gamma\to\PP^2$ from a general point of $\Gamma$})
=\frac{1}{2}(d-2)(d-3)-g
\end{align*}
Then we obtain the formula in the statement after plugging in the numbers to the equation
\begin{align*}
k\zeta+\Pi^*(-2kH+\ee_*\qq^*c_1(V)+ \ee_*E)
=&k\zeta-2k\Pi^*H+\left\{(r+k)\Pi^*H-\frac{2r}{d}\Pi^*D\right\}+m\Pi^*D
\\%[-0.1\baselineskip] %%%%%%%%%%%%%%%%%%%%%%%%%%%%%%%%%%%%%%%%%%%%%%%%%%
=&k\zeta+(r-k)\Pi^*H-\left(\frac{2r}{d}-m\right)\Pi^*D.
\end{align*}

The proof for the second case is similar to the first case except $\ggg=\text{Id}_{\overline{\Uu_0}}$, and hence $m=0$.
\end{proof}

\begin{Prop}
\label{space2}
Let $f:X=\Bl_{\Gamma} \PP^3\to \PP^3$ be the blow-up of the projective space $\PP^3$ along a smooth nondegenerate curve $\Gamma$.
Then there exists an irreducible and reduced effective divisor
\[
[\breve\Cc]
\sim k\zeta-k\Pi^*H+b\Pi^*D
\ \ \text{for some $b\geq k-2$}
\]
on $\PP(T_X)$ which satisfies \eqref{condition}.
Here, $k>0$ is the degree of the dual curve of a plane curve which is the image of $\Gamma$ under the projection from a general point of $\PP^3$, $H=f^*\Oo_{\PP^3}(1)$, and $D$ is the exceptional divisor of $f: X\to \PP^3$.
\end{Prop}

\begin{proof}
Let $\Kk$ be the family of {\color{black} unbendable} rational curves on $X$ containing the strict transforms of general lines on $\PP^3$ meeting $\Gamma$.
Note that $N_{\ell|X}\cong\Oo_{\PP^1}(1)\oplus\Oo_{\PP^1}$ for general $[\ell]\in\Kk$.
We denote by $\Cc_x$ the VMRT of $\Kk$ at $x\in X$, and $\breve\Cc$ the total dual VMRT on $\PP(T_X)$ associated to $\Kk$.

Let $x\in X\backslash D$ and $z=f(x)$.
We define $C_z\subseteq \PP(\Omega_{\PP^3}|_z)$ to be the plane curve parametrizing the tangent directions of lines on $\PP^3$ passing through $z$ and a point of $\Gamma$, and denote by $\breve C_z\subseteq\PP(T_{\PP^3}|_z)$ the projectively dual curve of $C_z$.
Notice that $C_z$ is isomorphic to the plane curve which is the image of $\Gamma$ under the projection $\PP^3\dashrightarrow\PP^2$ from $z$.
By the natural isomorphism $\PP(\Omega_X|_x)\cong\PP(\Omega_{\PP^3}|_z)$, we have the identifications
$\Cc_x\cong C_z$ and $\breve\Cc|_x\cong \breve C_z$ for general $x\in X\backslash D$.
As $\Cc_x$ is irreducible and reduced, so is $\breve\Cc|_x$ for general $x\in X$, and hence $\breve\Cc$ is irreducible and reduced. 

Let
\[
[\breve\Cc]\sim k\zeta+a\Pi^*H+b\Pi^*D
\]
for some $k>0$ and $a,\,b\in\ZZ$.
Then $k$ is equal to the degree of $\breve C_z$.

Let $[\ell]$ be a general member of $\Kk$ and $\widetilde{\ell}\subseteq\PP(T_X|_{\ell})\subseteq\PP(T_X)$ be its minimal section.
Since the members of the family $\widetilde{\Kk}$ of rational curves on $\PP(T_X)$ containing $[\widetilde{\ell}]$ dominates $\breve\Cc$, we have
\[
\dim\widetilde{\Kk}
=3
\geq -K_{\breve\Cc}.\widetilde{\ell}+\dim\breve\Cc-|\Aut(\PP^1)\,|
=(3\zeta-(k\zeta+a\Pi^*H+b\Pi^*D)).\widetilde{\ell}+1.
\]
We note that the dimension of $\Kk$ is 3 and $\widetilde\ell$ is uniquely determined by $\ell$, and hence we have $\dim\widetilde \Kk=3$.
Therefore, $b\geq -a-2$. 

Let $\elll_0$ be a general line on $\PP^3$ and $\ell_0$ be its strict transform on $X$.
Note that $N_{\ell_0|X}\cong N_{\elll_0|\PP^3} \cong \Oo_{\PP^1}(1)\oplus\Oo_{\PP^1}(1)$.
Let $\widetilde{\ell_0}\subseteq\PP(T_X|_{\ell_0})\subseteq\PP(T_X)$ be the section of $\ell_0$ corresponding to a quotient $T_X|_{\ell_0}\to N_{{\ell_0}|X}\to \Oo_{\PP^1}(1)$.
We will complete the proof by showing that
\[
\breve\Cc.\widetilde{\ell_0}
=k+(aH+bD).\ell_0
=k+a
=0,
\]
and it is sufficient to show that $\breve\Cc\cap\widetilde{\ell_0}\neq\emptyset$ if and only if $\widetilde{\ell_0}\subseteq \breve\Cc$.

Assume that $\breve\Cc\cap\widetilde{\ell_0}\neq\emptyset$.
That is, $\widetilde{\ell_0}|_x\in\breve\Cc|_x$ for some $x\in \ell_0$.
Then it is enough to show that $\widetilde{\ell_0}|_x\in\breve\Cc|_x$ for all $x\in\ell_0$.
Let $\widetilde{\elll_0}\subseteq\PP(T_{\PP^3}|_{\elll_0})\subseteq\PP(T_{\PP^3})$ be the image of $\widetilde{\ell_0}$ under the isomorphism $\PP(T_X|_{\ell_0})\xrightarrow{\simeq}\PP(T_{\PP^3}|_{\elll_0})$.
So it is reduced to prove the following claim.

\smallskip
\noindent \underline{Claim.}
$\widetilde{\elll_0}|_z\in\breve C_z$ for some $z\in \elll_0$ implies that $\widetilde{\elll_0}|_z\in\breve C_z$ for all $z\in \elll_0$.
\smallskip

Notice that $\widetilde{\elll_0}$ is a section of $\elll_0$ corresponding to a quotient $T_{\PP^3}|_{\elll_0}\to N_{{\elll_0}|{\PP^3}}\to \Oo_{\PP^1}(1)$.
Let $P$ be a plane containing $\elll_0$ on $\PP^3$.
Then it determines a quotient
\begin{equation}
\label{the_quotient}
T_{\PP^3}|_{\elll_0}\to N_{\elll_0|\PP^3}\to N_{P|\PP^3}|_{\elll_0}\cong \Oo_{\PP^1}(1),
\end{equation}
and vice versa.
That is, given $\elll_0$ on $\PP^3$, there is a $1$-to-$1$ correspondence between a plane $P$ containing $\elll_0$ and a section $\widetilde{\elll_0}$ of $\elll_0$ corresponding to a quotient $T_{\PP^3}|_{\elll_0}\to N_{\elll_0|\PP^3}\to\Oo_{\PP^1}(1)$, both are parametrized by $\mathbb P^2$. 

Let $z\in l_0$. As the kernel of the quotient \eqref{the_quotient} is $T_P|_{\elll_0}$, the line $L_z:=\PP(\Omega_P|_z) \subseteq \PP(\Omega_{\PP^3}|_z)$ parametrizes the tangent directions of $P$ at $z$.
Thus the line $L_z$ is tangent to $C_z\subseteq \PP(\Omega_{\PP^3}|_z)$ if and only if the point $\PP(N_{P|\PP^3}|_z)$ is contained in $\breve C_z\subseteq \PP(T_{\PP^3}|_z)$ as $L_z=\PP(\Omega_P|_z)$ and $\PP(N_{P|\PP^3}|_z)$ are projectively dual.
This is equivalent to say that $P$ is tangent to $\Gamma$ on $\PP^3$ because $C_z$ and $L_z$ are linear projections of $\Gamma$ and $P$ respectively after identifying $\PP^3\dashrightarrow\PP(\Omega_{\PP^3}|_z)$ to the projection from $z$.
If we choose a plane $P$ that contains $l_0$ and tangents to $\Gamma$, then any point $z\in l_0$ satisfies the above.
Therefore, we have the claim.
\end{proof}

\begin{Theorem}
\label{space-cor}
Let $f:X=\Bl_\Gamma\PP^3\to\PP^3$ be the blow-up of the projective space $\PP^3$ along a smooth nondegenerate curve $\Gamma$.
Assume that $\Gamma$ has at most a finite number of quadrisecant lines on $\PP^3$.
Then $T_X$ is big if and only if $\Gamma$ is a twisted cubic curve.
\end{Theorem}

\begin{proof}
Let $d$ and $g$ be the degree and genus of $\Gamma$.
If $d=3$, then $T_X$ is big by Remark~\ref{divisor_from_projective_space}.
Otherwise, if $d\geq 4$, then let $\breve\Cc_1$ and $\breve\Cc_2$ be the total dual VMRTs on $\PP(T_X)$ respectively given by Proposition~\ref{space2} and Proposition~\ref{space}.
Then, for general $x\in X$, $\breve\Cc_1|_x$ is projectively dual to a plane curve of degree $d$ with $\delta=\frac{1}{2}(d-1)(d-2)-g$ nodes, hence the degree of $\breve\Cc_1|_x$ is given by
%\vspace{-0.3em} %%%%%%%%%%%%%%%%%%%%%%%%%%%%%%%%%%%%%%%%%%%%%%%%%%
\[
d(d-1)-2\delta
=d(d-1)-2\cdot\left(\frac{1}{2}(d-1)(d-2)-g\right)
=2d+2g-2.
\vspace{-0.5em} %%%%%%%%%%%%%%%%%%%%%%%%%%%%%%%%%%%%%%%%%%%%%%%%%%
\]
So $\breve\Cc_1$ and $\breve\Cc_2$ satisfy
\begin{align*}
[\breve\Cc_1]
&\sim (2d+2g-2)\zeta-(2d+2g-2)\Pi^*H+\left((2d+2g-4)+c\right)\Pi^*D\\
[\breve\Cc_2]
&\sim \left(\frac{1}{2}(d-1)(d-2)-g\right)\zeta+(d+g-1)\Pi^*H-\left((d-1)-\frac{1}{2}(d-2)(d-3)+g\right)\Pi^*D
\end{align*}
for some $c\geq 0$ where $H=f^*\Oo_{\PP^3}(1)$ and $D$ is the exceptional divisor of $f: X\to\PP^3$ (the second formula is valid even in the case $d=4$ and $g=1$ where $\Gamma$ has no trisecant line on $\PP^3$).
Then they yield an effective divisor
\vspace{-0.3em} %%%%%%%%%%%%%%%%%%%%%%%%%%%%%%%%%%%%%%%%%%%%%%%%%%
\[
[\breve\Cc_1]+2[\breve\Cc_2]
=k\zeta+\left(
(d-1)(d-4)+c
\right)\Pi^*D
\]
on $\PP(T_X)$ satisfying \eqref{condition} for some $k>0$ and $c\geq 0$.
Thus by Lemma \ref{repeat}, $\zeta$ is not big on $\PP(T_X)$.
\end{proof}

\begin{Rmk}
\label{blow-ups_of_projective_space}
Let $X$ be a Fano threefold given by the blow-up $f:X=\Bl_{\Gamma}\PP^3\to \PP^3$ of the projective space $\PP^3$ along a smooth nondegenerate curve $\Gamma$ of degree $d$ and genus $g$.
Then $\Gamma$ has no quadrisecant line on $\PP^3$.
Indeed, if $\Gamma$  has a quadrisecant line, then $X$ is not Fano because the strict transform $\ell$ of a quadrisecant line satisfies $-K_X\cdot\ell=0$. 
\begin{center}
\vspace{0.7em}
\begin{tabular}{c|c|c|c|c|c|c|c|c}
\hline
No. in \cite[Table~2]{MM} &
4 & 9 & 12 & 15 & 17 & 19 & 22 & 25 \\
\hline
$(d,g)$ &
$(9,10)$ & $(7,5)$ & $(6,3)$ & $(6,4)$ & $(5,1)$ & $(5,2)$ & $(4,0)$ & $(4,1)$ \\
\hline
\end{tabular}
\vspace{0.7em}
\end{center} 
By Theorem~\ref{space-cor}, $T_X$ is not big in the above cases.
\end{Rmk}

{
\begin{Prop}
\label{quadric2}
Let $f:X=\Bl_\Gamma Q\to Q$ be the blow-up of a smooth quadric threefold $Q\subseteq\PP^4$ along a smooth curve $\Gamma$ of degree $d$.
Assume that $\Gamma$ has at most a finite number of trisecant lines on $Q$.
Then there exists an irreducible and reduced effective divisor
\[
[\breve\Cc]
\sim
d\zeta
+\left(m-2\right)\Pi^*D
\]
on $\PP(T_X)$ which satisfies \eqref{condition}.
Here, $m\geq 0$ is the number of secant lines of $\Gamma$ on $Q$ passing through a general point of $\Gamma$, and $D$ is the exceptional divisor of $f:X\to Q$.
\end{Prop}
\begin{proof}
We continue to use the notation in the proof of Proposition~\ref{space}.
The proof is similar to the proposition.
Let $\Kk$ be the family of {\color{black} unbendable} rational curves on $X$ containing the strict transforms of lines meeting $\Gamma$ on $Q$, and $\breve\Cc$ be the total dual VMRT on $\PP(T_X)$ associated to $\Kk$.
From the diagram
\vspace{-0.3em} %%%%%%%%%%%%%%%%%%%%%%%%%%%%%%%%%%%%%%%%%%%%%%%%%%
\[\xymatrix @R=2pc @C=3pc {
\overline\Kk \ar@{=}[d] & \overline\Uu \ar[r]^{\ee} \ar[l]_{\qq} \ar[d]_\ggg &
X \ar[d]^{f} \\
\overline\Kk  &
\overline{\Uu_0}\cong \PP_{\overline\Kk}(V) \ar[r]_{\ee_0} \ar[l]^{\qq_0} \ar@{-->}[ur]&
Q,
}
\vspace{-0.7em} %%%%%%%%%%%%%%%%%%%%%%%%%%%%%%%%%%%%%%%%%%%%%%%%%%
\]
we have the formula
\[
[\breve\Cc]
\sim k\zeta+\Pi^*\ee_*(K_{\overline\Uu/\overline\Kk})
=k\zeta
-2\Pi^*\ee_*\xi
+\Pi^*\ee_*\qq^*c_1(V)
+\Pi^*\ee_*E
\]
for $k=\deg(\ee)$.

As a general hyperplane section meets a general line at $1$ point on $Q$, we take $V={\qq}_*{\ee}^*H$ so that $\xi={\ee}^*H$ for the tautological class $\xi=\Oo_{\PP_{\overline{\Kk}}}(1)$.
Then $\ee_*\xi=kH$.
Let
\[
\ee_*\qq^*c_1(V)=aH+bD\ \ 
\text{and}\ \ 
\ee_*E=mD
\]
for some $a,\,b\in\ZZ$ and $m\geq 0$.
In this case, $H^3=2$, $H^2.D=0$, and $H.D^2=-d$.
Thus from the same calculation in the proof of Proposition~\ref{space},
%\vspace{0.1em} %%%%%%%%%%%%%%%%%%%%%%%%%%%%%%%%%%%%%%%%%%%%%%%%%%
\[
a
=\frac{1}{2}{c_1(V)}^2
=k+\frac{1}{2}r,\ \ 
b
=-\frac{1}{2}\qq^*c_2(V).\ee^*D
=-\frac{1}{d}r
%\vspace{0.1em} %%%%%%%%%%%%%%%%%%%%%%%%%%%%%%%%%%%%%%%%%%%%%%%%%%
\]
for $r:=c_2(V)$ as a general line is chosen to meet $\Gamma$ at 1 point on $Q$ so that $\ee^*D\equiv \xi$ in $\text{N}^1(\overline{\Uu}/\overline{\Kk})$.

Thus we obtain the formula in the statement using the following correspondences.
\begin{align*}
k
&=(\text{the number of lines on $Q$ joining $\Gamma$ and a general point of $Q$})
=d
\\[0.1\baselineskip]
r
&=(\text{the number of lines on $Q$ meeting $\Gamma$ and contained in a general hyperplane section of $Q$})
=2d
\\[0.1\baselineskip]
m
&=(\text{the number of secant lines of $\Gamma$ on $Q$ passing through a general point of $\Gamma$})
\qedhere
\end{align*}
\end{proof}
}

{\begin{Rmk}
\label{blow-ups_of_quadric_threefold2}
Let $f:X=\Bl_{\Gamma}Q\to Q$ be the blow-up of a smooth quadric threefold $Q\subseteq\PP^4$ along a smooth curve $\Gamma$.
We denote by $H=f^*\Oo_Q(1)$, and $D$ the exceptional divisor of $f: X\to Q$.

If $\Gamma$ is a quartic elliptic curve (No.~23 in \cite[Table~2]{MM}) given by the intersection of $Q$ with $A\in|\Oo_{\PP^4}(1)|$ and $B\in|\Oo_{\PP^4}(2)|$ on $\PP^4$, then the number of secant lines of $\Gamma$ on $Q$ passing through a general point of $\Gamma$ is $m=2$.

Indeed, in No.~23.a, a secant line of $\Gamma$ on $Q$  corresponds to a rule in the quadric surface $Q\cap A\subseteq A\cong \PP^3$ and $\Gamma=Q\cap A\cap B$ corresponds to a $(2,2)$-divisor on the quadric surface $Q\cap A$.
Then the number of secant lines of $\Gamma$ on $Q$ passing through a general point of $\Gamma$ is $m=2$.
In No.~23.a, the locus of secant lines of $\Gamma$ is two disjoint lines which is parametrized by two rulings in $Q\cap A$.
In No.~23.b, the locus of secant lines of $\Gamma$ is a double line which is also parametrized by the ruling in a quadric cone surface $Q\cap A$.
So we have the same $m=2$.

Let $\breve\Cc$ be the total dual VMRT on $\PP(T_X)$ given by Proposition~\ref{quadric2}.
Then it yields an effective divisor
\[
[\breve\Cc]
\sim 4\zeta
\]
on $\PP(T_X)$ satisfying \eqref{condition}.
Thus by Lemma \ref{repeat}, $\zeta$ is not big on $\PP(T_X)$.

If $\Gamma$ is a smooth conic (No.~29 in \cite[Table~2]{MM}), then there is no secant line of $\Gamma$ on $Q$ because the plane spanned by $\Gamma$ and $Q$ have $\Gamma$ as their intersection.

Let $\breve\Cc_1$ be the total dual VMRT on $\PP(T_X)$ given by Proposition~\ref{quadric2}.
Then it yields an effective divisor
\[
[\breve\Cc_1]
\sim 2\zeta-2\Pi^*D
\]
on $\PP(T_X)$.
Together with the total dual VMRT $\breve\Cc_2$ on $\PP(T_X)$ given by Proposition~\ref{quadric},
we can write a multiple of $\zeta$ by a positive linear combination of effective divisors on $\PP(T_X)$ as
\[
4\zeta=[\breve\Cc_1]+[\breve\Cc_2]+2\Pi^*H.
\]
Thus $\zeta$ is big on $\PP(T_X)$.
That is, $T_X$ is big.
\end{Rmk}}

\section{Conic Bundles with Non-empty Discriminant}

In this section, we disprove the bigness of $T_X$ in the case when a Fano threefold $X$ admits a standard conic bundle structure $\pi: X\to \PP^2$ with non-empty discriminant $\Delta\subseteq\PP^2$ of degree $d=\deg\Delta$.
When the discriminant is empty, i.\,e., $X$ admits a $\PP^1$-fibration, we have proved the bigness of $T_X$ from a number of remarks in previous sections except No. 24 in \cite[Table~2]{MM}, which admits another conic bundle structure with non-empty discriminant (see \ref{MM:2-24}).

\subsection{Degree 3}
There are two deformation types of standard conic bundle $\pi:X\to\PP^2$ with $d=3$:\linebreak
No. 20 and 24 in \cite[Table~2]{MM}.

\subsubsection{No. 20}\label{MM:2-20}
Let $f:X\to V_5$ be the blow-up of the smooth del Pezzo threefold $V_5$ of degree $5$ along a smooth curve $\Gamma$ of degree $3$ and genus $0$.
\[\xymatrix @R=2pc @C=2pc {
X=\Bl_\Gamma V_5 \ar[d]_{\pi=\varphi_{|H-D|}} \ar[r]^{f} & V_5 \\ \PP^2
}
\]
We denote by $H=f^*\Oo_{V_5}(1)$, and $D$ the exceptional divisor of $f: X\to V_5$.
Then
\[
-K_X\sim 2H-D= H+h,\ \ 
h\sim H-D,\ \ 
D\sim H-h.
\]

Let $\breve\Cc_1$ be the total dual VMRT on $\PP(T_X)$ associated to the family of fibers of $\pi: X\to\PP^2$.
Then we know from Proposition~\ref{conic_bundle_formula} that
\[
[\breve\Cc_1]
\sim \zeta+\Pi^*(K_X-\pi^*K_{\PP^2})
=\zeta+\Pi^*((-H-h)-(-3h))
=\zeta-\Pi^*H+2\Pi^*h.
\]

On the other hand, let $\breve\Cc_2$ be the total dual VMRT on $\PP(T_X)$ associated to the family of the strict transforms of lines on $V_5$.
Then, by Proposition~\ref{V5},
\[
[\breve\Cc_2]
\sim 3\zeta-\Pi^*H+3\Pi^*D
=3\zeta-\Pi^*H+3\Pi^*(H-h)
=3\zeta+2\Pi^*H-3\Pi^*h.
\]

So $5\zeta$ can be expressed as
\[
2[\breve\Cc_1]+[\breve\Cc_2]=5\zeta+\Pi^*h
\]
on $\PP(T_X)$ satisfying \eqref{condition}.
Thus by Lemma \ref{repeat}, $\zeta$ is not big.

\subsubsection{No. 24}\label{MM:2-24}
Let $X$ be a $(1,2)$-divisor on $\PP^2\times\PP^2$.
Then $X$ is not only a conic bundle $\pi_1:X\to\PP^2$ but also a ruled variety $\pi_2: X=\PP(E)\to \PP^2$ for some vector bundle $E$ of rank $2$ on $\PP^2$.
\[\xymatrix @R=2pc @C=2pc {
X \ar[d]_{\pi_1} \ar[r]^{\pi_2} & \PP^2 \\ \PP^2
}
\]
We denote by $h_i=\pi_i^*\Oo_{\PP^2}(1)$ for $i=1,\,2$.
Then $-K_X\sim 2h_1+h_2$.

Let $\breve\Cc_i$ be the total dual VMRT on $\PP(T_X)$ associated to the family of fibers of $\pi_i: X\to\PP^2$ for $i=1,\,2$.
Then we know from Proposition~\ref{conic_bundle_formula} that
\begin{align*}
[\breve\Cc_1]
&\sim\zeta+\Pi^*(K_X-\pi_1^* K_{\PP^2})
=\zeta+\Pi^*((-2h_1-h_2)-(-3h_1))
=\zeta+\Pi^*h_1-\Pi^*h_2,\\
[\breve\Cc_2]
&\sim\zeta+\Pi^*(K_X-\pi_2^* K_{\PP^2})
=\zeta+\Pi^*((-2h_1-h_2)-(-3h_2))
=\zeta-2\Pi^*h_1+2\Pi^*h_2.
\end{align*}

So $3\zeta$ can be expressed as
\[
2[\breve\Cc_1]+[\breve\Cc_2]
=3\zeta
\]
on $\PP(T_X)$ satisfying \eqref{condition}.
Thus by Lemma \ref{repeat}, $\zeta$ is not big.

\subsection{Degree 4}
There are two deformation types of standard conic bundles $\pi:X\to\PP^2$ with $d=4$:\linebreak
No. 13 and 18 in \cite[Table~2]{MM}.

\subsubsection{No. 13}\label{MM:2-13}
Let $f: X=\Bl_\Gamma Q\to Q$ be the blow-up of a smooth quadric threefold $Q\subseteq\PP^4$ along a smooth curve $\Gamma$ of degree $6$ and genus $2$ \cite[Section 5.2]{BB}.
\[\xymatrix @R=2pc{
X=\Bl_\Gamma Q \ar[d]_{\pi=\varphi_{|2H-D|}} \ar[r]^{\ \ \ \ \ f}
& Q\\
\PP^2 & &
}
\]
We denote by $h=\pi^*\Oo_{\PP^2}(1)$, $H=f^*\Oo_Q(1)$, and $D$ the exceptional divisor of $f: X\to Q$.
Then
\[
-K_X\sim 3H-D= H+h,\ \ 
h\sim 2H-D,\ \ 
D\sim 2H-h.
\]

Let $\breve\Cc_1$ be the total dual VMRT on $\PP(T_X)$ associated to the family of fibers of $\pi: X\to\PP^2$.
Then we know from Proposition~\ref{conic_bundle_formula} that
\[
[\breve\Cc_1]
\sim \zeta+\Pi^*(K_X-\pi^*K_{\PP^2})
=\zeta+\Pi^*((-H-h)-(-3h))
=\zeta-\Pi^*H+2\Pi^*h.
\]

On the other hand, let $\breve\Cc_2$ be the total dual VMRT on $\PP(T_X)$ associated to the family of the strict transforms of lines on $Q$.
Then, by Proposition~\ref{quadric},
\[
[\breve\Cc_2]
\sim 2\zeta-2\Pi^*H+2\Pi^*D
=2\zeta-2\Pi^*H+2\Pi^*(2H-h)
=2\zeta+2\Pi^*H-2\Pi^*h.
\]

So $3\zeta$ can be expressed as
\[
[\breve\Cc_1]+[\breve\Cc_2]=3\zeta+\Pi^*H
\]
on $\PP(T_X)$ satisfying \eqref{condition}.
Thus by Lemma \ref{repeat}, $\zeta$ is not big.

\subsubsection{No. 18}\label{MM:2-18}
Let $f: X\to \PP^2\times \PP^1$ be the double cover branched over a $(2,2)$-divisor $B$ on $\PP^2\times\PP^1$.
\[\xymatrix @R=2pc @C=3.5pc {
& X \ar[dl]_{\pi} \ar[dr]^{p} \ar[d]^f & \\
\PP^2 & \PP^2\times \PP^1 \ar[l]
\ar[r]
& \PP^1
}\]
We denote by $h=\pi^*\Oo_{\PP^2}(1)$ and $H=p^*\Oo_{\PP^1}(1)$.
Then $-K_X\sim f^*(-K_{\PP^2\times\PP^1}-\frac{1}{2}B)=H+2h$.

Let $\breve\Cc_1$ be the total dual VMRT on $\PP(T_X)$ associated to the family of fibers of the conic bundle $\pi: X\to\PP^2$.
Then we know from Proposition~\ref{conic_bundle_formula} that
\[
[\breve\Cc_1]
\sim \zeta+\Pi^*(K_X-\pi^*K_{\PP^2})
=\zeta+\Pi^*((-H-2h)-(-3h))
=\zeta-\Pi^*H+\Pi^*h.
\]

Note that $p:X\to \PP^1$ is a del Pezzo fibration whose general fiber $X_t$ is isomorphic to a smooth quadric surface whereas a singular fiber is isomorphic to a quadric cone (cf. \cite[p.\,109]{MM2}).
Let
\begin{itemize}
\item $\Kk$ be the family of rational curves on $X$ containing a rule $\ell$ of $X_t\cong\PP^1\times\PP^1$,
\end{itemize}
and $\breve\Cc_2$ be the total dual VMRT on $\PP(T_X)$ associated to $\Kk$, whose linear class is given by
\[
[\breve\Cc_2] \sim k\zeta + c\Pi^*H + d\Pi^*h
\]
for some $k>0$ and $c,\,d\in\ZZ$.
Because there are two families of rules in each smooth fiber $X_t$, $\Kk$ may parametrize both of the families in $X_t$, or parametrize only one of the families in $X_t$.

The second case does exist: if $\Kk$ parametrizes only the lines of one ruling, then we can choose a section $s$ of the fibration.
The curves in $\Kk$ meeting the section $s$ form a divisor $D \subset X$ such that the restriction to the general fiber is $\Oo_{\PP^1\times\PP^1}(1,0)$.
This contradicts to $\rho(X/\PP^1)=1$.

When $\Kk$ parametrizes both of the families of rules in each smooth fiber $X_t$, it is obvious that $k=2$.
Let $\overline{\Kk}$ be the normalization of the closure of $\Kk$ as a subscheme of $\text{Chow}(X)$ and $\qq:\overline{\Uu}\to\overline{\Kk}$ be the normalization of the universal family with the evaluation morphism $\ee:\overline{\Uu}\to X$.
Note that $\qq: \overline{\Uu}\to\overline{\Kk}$ is a $\PP^1$-fibration over $\overline{\Kk}$, and from its construction, there exists a rank $2$ vector bundle $V$ on $\overline{\Kk}$ such that $\overline{\Uu}=\PP_{\overline{\Kk}}(V)$.
Indeed, from a fixed embedding $X\to \PP^3\times\PP^1$ with the second projection being a quadric fibration, $V$ is given by the pull-back of the universal bundle over the Grassmannian $\mathrm{Gr}(2,4)$ via the induced morphism $\overline{\Kk}\to\mathrm{Gr}(2,4)\times \PP^1\to\mathrm{Gr}(2,4)$.
Then, by Proposition~\ref{conic_bundle_formula}, the linear class of $\breve\Cc_2$ satisfies
\[
[\breve\Cc_2]
\sim 2\zeta+\Pi^*\ee_*(K_{\overline{\Uu}/\overline{\Kk}}).
\]
We take $V=\qq_*\ee^*h$ so that $\xi=\ee^*h$ for the tautological class $\xi=\Oo_{\PP_{\overline{\Kk}}(V)}(1)$.
Then
\[
\ee_*(K_{\overline{\Uu}/\overline{\Kk}})
=\ee_*(-2\xi+\qq^*c_1(V))
=-4h+\ee_*\qq^*c_1(V).
\]

Let
\[
\ee_*\qq^*c_1(V)=aH+bh
\]
for some $a,\,b\in\ZZ$.
Then
\[
2a
=aH.h^2
=\ee_*\qq^*c_1(V).h^2
=\qq^*c_1(V).\xi^2
=\qq^*c_1(V)(\qq^*c_1(V)\xi-\qq^*c_2(V))
={c_1(V)}^2.
\]
From ${c_1(V)}^2-c_2(V)=\xi^3=2h^3=0$, we have $a=\frac{1}{2}{c_1(V)}^2=\frac{1}{2}c_2(V)$.
We can find $r:=c_2(V)$ as the number of indeterminate points of a rational section of $\qq:\overline{\Uu}\to\overline{\Kk}$ linearly equivalent to $\xi=\ee^*h$.

We fix a general member $h_0\in|h|$.
Then $h_0=\pi^*{L_0}$ for some line $L_0$ on $\PP^2$.
Recall that a smooth fiber $X_t$ is isomorphic to the double cover $\pi_t: X_t\to\PP^2$ branched over a conic $C_t$ and $h_0\cap X_t={\pi_t}^{-1}(L_0)$.
Then $h_0\cap X_t$ is a union of two rules in the fiber $X_t$ if and only if $L_0$ is tangent to $C_t$ on $\PP^2$.
Here, a singular fiber (two-dimensional quadric cone) of $p$ is not affected due to a general choice of $h_0$: a singular fiber is obtained by a double cover $\PP^2$ branched over a union of two lines.
Since we have only finitely many singular fibers of $p$, we can choose a line $h_0$ which is not one of the lines of the branch loci of the singular fibers of $p$.

As the condition of being tangent to $L_0$ corresponds to a hypersurface of degree $2$ in the space of conics, there are only four conics $C_t$ for $t=1,\,2,\,3,\,4$ such that $h_0\cap X_t={\pi_t^{-1}}(L_0)$ is a union of two rules.
Thus the total number of $\overline{q}$-fibers contained in $\xi_0=\ee^*h_0$ on $\overline{\Uu}$ is $r=2\times 4=8$, and so $a=4$.

On the other hand,
\begin{align*}
2b
&=bH.h^2
=\ee_*\qq^*c_1(V).H.h
=\qq^*c_1(V).\xi.\ee^*H
=(\xi^2+\qq^*c_2(V)).\ee^*H
=\xi^2.\ee^*H
=\ee^*(H.h^2)
=2H.h^2
=4
\end{align*}
as $\ell$ does not meet $H$ on $X$ for general $[\ell]\in\overline{\Kk}$ so that $\ee^*H\equiv 0$ in $\text{N}^1(\overline{\Uu}/\overline{\Kk})$.
Thus $b=2$, and we have
\[
[\breve\Cc_2]
\sim 2\zeta+4\Pi^*H-2\Pi^*h.
\]

So $3\zeta$ can be expressed as
\[
2[\breve\Cc_1]+[\breve\Cc_2]=3\zeta+2\Pi^*H
\]
on $\PP(T_X)$ satisfying \eqref{condition}.
Thus by Lemma \ref{repeat}, $\zeta$ is not big.

\subsection{Degree 5}
There are two deformation types of standard conic bundles $\pi:X\to\PP^2$ with $d=5$:\linebreak
$\pi$ corresponds to odd (No. 11), or even theta characteristic (No. 9 in \cite[Table~2]{MM}).

\subsubsection{No. 11}\label{MM:2-11}
If $\pi$ corresponds to odd theta-characteristic, then $X$ is isomorphic to the blow-up $\Bl_\Gamma V_3$ of a smooth cubic threefold $V_3\subseteq \PP^4$ along a line $\Gamma$.
We know that $T_X$ is big only if $T_{V_3}$ is big by Lemma~\ref{bigness_of_blow-up}, but $T_{V_3}$ cannot be big by \cite[Theorem~1.4]{HLS}.
Thus $T_X$ is not big (see Remark~\ref{blow-ups_of_del_pezzo_manifold}).

\subsubsection{No. 9}\label{MM:2-9}
If $\pi$ corresponds to even theta-characteristic, then $X$ is isomorphic to the blow-up $f: X=\Bl_\Gamma\PP^3\to\PP^3$ of the projective space $\PP^3$ along a smooth curve $\Gamma$ of degree $7$ and genus $5$ \cite[Section 5.1]{BB}.
\[\xymatrix @R=2pc{
X=\Bl_\Gamma \PP^3\ar[d]_{\pi=\varphi_{|3H-D|}} \ar[r]^{\ \ \ \ \ f}
& \PP^3 & \\
\PP^2 & &
}\]
We denote by $h=\pi_*\Oo_{\PP^2}(1)$, $H=f^*\Oo_{\PP^3}(1)$, and $D$ the exceptional divisor of $f: X\to \PP^3$.
Then
\[
-K_X\sim 4H-D= H+h,\ \ 
h\sim 3H-D,\ \
D\sim 3H-h.
\]

Let $\breve\Cc$ be the total dual VMRT on $\PP(T_X)$ associated to the family of the strict transforms of secant lines of $\Gamma$ on $\PP^3$.
Then, by Proposition~\ref{space},
\[
[\breve\Cc]
\sim
10\zeta+11\Pi^*H-\Pi^*D
=10\zeta+8\Pi^*H+\Pi^*h,
\]
and $\breve\Cc$ satisfies \eqref{condition}.
Thus by Lemma \ref{repeat}, $\zeta$ is not big on $\PP(T_X)$ (see also Remark~\ref{blow-ups_of_projective_space}).

%\pagebreak
\subsection{Degree 6}
There are two deformation types of standard conic bundles $\pi:X\to\PP^2$ with $d=6$:\linebreak
No. 6 and 8 in \cite[Table~2]{MM}.

\subsubsection{No. 6}\label{MM:2-6}
We first treat the case where $X$ is isomorphic to a $(2,2)$-divisor on $\PP^2\times \PP^2$ (No. 6.a).
%\vspace{-0.5em} %%%%%%%%%%%%%%%%%%%%%%%%%%%%%%%%%%%%%%%%%%%%%%%%%%
\[\xymatrix @R=2pc @C=2pc {
X \ar[d]_{\pi_1} \ar[r]^{\pi_2} & \PP^2 \\ \PP^2
}
%\vspace{-0.2em} %%%%%%%%%%%%%%%%%%%%%%%%%%%%%%%%%%%%%%%%%%%%%%%%%%
\]
We denote by $h_i=\pi_i^*\Oo_{\PP^2}(1)$ for $i=1,\,2$.
Then $-K_X\sim h_1+h_2$.

Let $\breve\Cc_i$ be the total dual VMRT on $\PP(T_X)$ associated to the family of fibers of the conic bundle $\pi_i: X\to\PP^2$ for $i=1,\,2$.
Then we know from Proposition~\ref{conic_bundle_formula} that
\begin{align*}
[\breve\Cc_1]
&\sim \zeta + \Pi^*(K_X-\pi_1^*K_{\PP^2})
= \zeta + (\Pi^*(-h_1-h_2) - \Pi^*(-3h_1))
= \zeta + 2\Pi^*h_1 - \Pi^*h_2,\\
[\breve\Cc_2]
&\sim \zeta + \Pi^*(K_X-\pi_2^*K_{\PP^2})
= \zeta + (\Pi^*(-h_1-h_2) - \Pi^*(-3h_2))
= \zeta - \Pi^*h_1 + 2\Pi^*h_2.
\end{align*}

So $2\zeta$ can be expressed as
\[
[\breve\Cc_1]+[\breve\Cc_2]=2\zeta+\Pi^*(h_1+h_2)
\]
on $\PP(T_X)$ satisfying \eqref{condition}.
Thus by Lemma \ref{repeat}, $\zeta$ is not big.

We next deal with the case where $X$ is isomorphic to the double cover of a $(1,1)$-divisor $W$ on $\PP^2\times \PP^2$ which is branched over a smooth member $B\in|-K_W|$ (No. 6.b).
%\vspace{-0.5em} %%%%%%%%%%%%%%%%%%%%%%%%%%%%%%%%%%%%%%%%%%%%%%%%%%
\[\xymatrix @R=2pc @C=3.5pc {
& X \ar[dl]_{\pi_1} \ar[dr]^{\pi_2} \ar[d]^f & \\
\PP^2 & W \ar[l] \ar[r]
& \PP^2
}
%\vspace{-0.2em} %%%%%%%%%%%%%%%%%%%%%%%%%%%%%%%%%%%%%%%%%%%%%%%%%%
\]
We denote by $h_i=\pi_i^*\Oo_{\PP^2}(1)$.
Then $-K_X\sim f^*(-K_{\PP^2\times\PP^2}-\frac{1}{2}B)=h_1+h_2$.

Let $\breve\Cc_i$ be the total dual VMRT on $\PP(T_X)$ associated to the family of fibers of the conic bundle $\pi_i: X\to\PP^2$ for $i=1,\,2$.
Then we know from Proposition~\ref{conic_bundle_formula} that
\begin{align*}
[\breve\Cc_1]
&\sim \zeta + \Pi^*(K_X- \pi_1^*K_{\PP^2})
=\zeta + (\Pi^*(-h_1-h_2) - \Pi^*(-3h_1))
=\zeta + 2\Pi^*h_1 - \Pi^*h_2,\\
[\breve\Cc_2]
&\sim \zeta + \Pi^*(K_X- \pi_2^*K_{\PP^2})
=\zeta + (\Pi^*(-h_1-h_2) - \Pi^*(-3h_2))
=\zeta - \Pi^*h_1 + 2\Pi^*h_2.
\end{align*}

So $2\zeta$ can be expressed as
\[
[\breve\Cc_1]+[\breve\Cc_2]=2\zeta+\Pi^*(h_1+h_2)
\]
on $\PP(T_X)$ satisfying \eqref{condition}.
Thus by Lemma \ref{repeat}, $\zeta$ is not big on $\PP(T_X)$.

\subsubsection{No. 8}\label{MM:2-8}
Let $f: X\to V_7$ be the double cover branched over a smooth member $B\in|-K_{V_7}|$ where $V_7$ is the blow-up $g:V_7=\Bl_z\PP^3\to \PP^3$ of the projective space $\PP^3$ at a point $z\in\PP^3$.
Note that $V_7$ is also given by the ruled variety $\phi: V_7=\PP(\Oo_{\PP^2}\oplus\Oo_{\PP^2}(1))\to\PP^2$ over $\PP^2$.
%\vspace{-0.5em} %%%%%%%%%%%%%%%%%%%%%%%%%%%%%%%%%%%%%%%%%%%%%%%%%%
\[\xymatrix @R=2pc @C=3.5pc{
& X \ar[r]^{f\ \ \ \ \ } \ar[rd]_{\pi=\varphi_{|H-D|}} & V_7=\Bl_z\PP^3 \ar[r]^{\ \ \ \ \ g} \ar[d]_\phi
 & \PP^3 \ar@{-->}[ld] \\
& & \PP^2 & &
}
\]
We denote by $H=(g\circ f)^*\Oo_{\PP^3}(1)$, $h=\pi^*\Oo_{\PP^2}(1)$, $D=f^{-1}(D')$ for the exceptional divisor $D'$ of $g: V_7\to \PP^3$, and $B_0=g(B)$ the image of $B$ on $\PP^3$.
Then
\[
-K_X\sim 2H-D=H+h,\ \ 
h\sim H-D,\ \ 
D\sim H-h.
\]

Note that $B_0$ is a quartic surface with one ordinary double point at the center of the blow-up $z\in\PP^3$.
Let $\Kk$ be the family of rational curves on $X$ which parametrizes the irreducible components of preimages $\ell$ of bitangent lines $l$ to $B_0$ on $\PP^3$, and
$\breve\Cc$ be the total dual VMRT on $\PP(T_X)$ associated to $\Kk$.

Let $\overline{\Kk}$ be the normalization of the closure of $\Kk$ as a subscheme of $\text{Chow}(X)$ and $\qq: \overline{\Uu}\to\overline{\Kk}$ be the normalization of the universal family with the evaluation morphism $\ee:\overline{\Uu}\to X$.
Note that $\qq: \overline{\Uu}\to\overline{\Kk}$ is a $\PP^1$-fibration over $\overline{\Kk}$, and from its construction, there exists a rank $2$ vector bundle $V$ on $\overline{\Kk}$ such that $\overline{\Uu}=\PP_{\overline{\Kk}}(V)$.
Indeed, from the morphism $X\to V_7\to \PP^3$, $V$ is given by the pull-back of the universal bundle over the Grassmannian $\mathrm{Gr}(2,4)$ via the induced morphism $\overline{\Uu}\to \mathrm{Gr}(2,4)$.
Then, by Proposition~\ref{conic_bundle_formula}, the linear class of $\breve\Cc$ satisfies
\[
[\breve\Cc]
\sim k\zeta+\Pi^*\ee_*(K_{\overline{\Uu}/\overline{\Kk}})
\]
for $k=\deg (\ee)$.
We take $V=\qq_*\ee^*H$ so that $\xi=\ee^*H$ for the tautological class $\xi=\Oo_{\PP_{\overline{\Kk}}(V)}(1)$.
Then
\[
\overline{e}_*(K_{\overline{\Uu}/\overline{\Kk}})
=\ee_*(-2\xi+\qq^*c_1(V))
=-2kH+\ee_*\qq^*c_1(V).
\]

Let
\[
\ee_*\qq^*c_1(V)=aH+bD
\]
for some $a,\,b\in\ZZ$.
Then
\[
2a
=aH^3
=\ee_*\qq^*c_1(V).H^2
=\qq^*c_1(V).\xi^2
=\qq^*c_1(V)(\qq^*c_1(V)\xi-\qq^*c_2(V))
={c_1(V)}^2.
\]
From ${c_1(V)}^2-c_2(V)=\xi^3=kH^3=2k$, we have $a=\frac{1}{2}{c_1(V)}^2=k+\frac{1}{2}c_2(V)$.
We can calculate $r:=c_2(V)$ as the number of indeterminate points of a rational section of $\qq:\overline{\Uu}\to\overline{\Kk}$ linearly equivalent to $\xi=\ee^*H$.

On the other hand,
\begin{align*}
2b
=bD^3
=\ee_*\qq^*c_1(V).D^2
=\qq^*c_1(V).(\ee^*D)^2
=0
\end{align*}
as a general bitangent line $l$ of $B_0$ does not pass through $z$ on $\PP^3$, its strict transform $\ell$ does not meet $D$ on $X$ for general $[\ell]\in\overline{\Kk}$ so that $\ee^*D\equiv 0$ in $\text{N}^1(\overline{\Uu}/\overline{\Kk})$.
Thus $b=0$.

In this case, $k=\deg(\ee)$ and $r=\#\{\text{$l$ is a bitangent line to $B_0$ on $\PP^3$}\,|\,l\subseteq P\}$ for general $P\in|\Oo_{\PP^3}(1)|$ correspond to
\begin{align*}
k
&=(\text{the number of bitangent lines of $B_4\subseteq\PP^3$ passing through a general point})
=12,
\\
r
&=2\times (\text{the number of bitangent lines of $B_4\cap\PP^2$ contained in a general hyperplane $\PP^2\subseteq \PP^3$})
=56
\end{align*}
(for the first equality, see \cite[Theorem~1.1]{HK}).
So $12\zeta$ can be expressed as
\[
[\breve\Cc]
\sim
12\zeta+\left(\frac{r}{2}-k\right)\Pi^*H
=12\zeta+16\Pi^*H
\]
on $\PP(T_X)$ satisfying \eqref{condition}.
Thus by Lemma \ref{repeat}, $\zeta$ is not big on $\PP(T_X)$.

\subsection{Degree 8}
There is one deformation type of standard conic bundle $\pi:X\to\PP^2$ with $d=8$:\linebreak
No. 2 in \cite[Table~2]{MM}.

\subsubsection{No. 2}\label{MM:2-2}
Let $X$ be the double cover $f: X\to \PP^2\times \PP^1$ branched over a $(4,2)$-divisor $B$ on $\PP^2\times \PP^1$.
\[\xymatrix @R=2pc @C=3.5pc {
& X \ar[dl]_{\pi} \ar[dr]^{p} \ar[d]^f & \\
\PP^2 & \PP^2\times \PP^1 \ar[l]
\ar[r]
& \PP^1
}
\]
We denote by $h=\pi^*\Oo_{\PP^2}(1)$ and $H=p^*\Oo_{\PP^1}(1)$.
Then $-K_X\sim f^*(-K_{\PP^2\times\PP^1}-\frac{1}{2}B)=H+h$.

Note that $p:X\to \PP^1$ is a del Pezzo fibration whose general fiber $X_t$ is isomorphic to the double cover $\pi_t: X_t\to \PP^2$ branched over a smooth plane quartic $B_t$.
Recall that $X_t$ is a del Pezzo surface of degree $2$ which is also isomorphic to the blow-up $\rho_t: X_t\to\PP^2$ of the projective plane $\PP^2$ at $7$ points in general position.
Therefore, we can conclude that $T_X$ is not big by Proposition~\ref{fibration}.


\bigskip


\begin{thebibliography}{OSW16}

\bibitem[Bat82]{Bat}
    V. V. Batyrev,
    \textit{Toroidal Fano 3-folds},
    Math. USSR Izvestija \textbf{19}(1) (1982),
    13--25.

\bibitem[BB13]{BB}
    M. Bernardara, M. Bolognesi,
    \textit{Derived categories and rationality of conic bundles},
    Compositio Math. \textbf{149}(11) (2013),
    1789--1817.

\bibitem[FN89]{FN}
    M. Furushima, N. Nakayama,
    \textit{The family of lines on the Fano threefold $V_5$},
    Nagoya Math. J. \textbf{116} (1989),
    111--122.

\bibitem[HL23]{HL}
    A. H\"{o}ring, J. Liu,
    \textit{Fano manifolds with big tangent bundle: a characterisation of $V_5$},
    Collect. Math. \textbf{74}(3) (2023),
    639--686.

\bibitem[HLS22]{HLS}
    A. H\"{o}ring, J. Liu, F. Shao,
    \textit{Examples of Fano manifolds with non-pseudoeffective tangent bundle},
    J. Lond. Math. Soc. (2) \textbf{106}(1) (2022),
    27--59.

\bibitem[Hsi15]{Hsi}
    J.-C. Hsiao,
    \textit{A remark on bigness of the tangent bundle of a smooth projective variety and $D$-simplicity of its section rings},
    J. Algebra Appl. \textbf{14}(7) (2015),
    1550098.

\bibitem[HK15]{HK}
    J.-M. Hwang, H. Kim,
    \textit{Varieties of minimal rational tangents on Veronese double cones},
    Algebr. Geom. \textbf{2}(2) (2015),
    176--192.

\bibitem[HM04]{HM}
    J.-M. Hwang, N. Mok,
    \textit{Birationality of the tangent map for minimal rational curves},
    Asian J. Math. \textbf{8}(1) (2004),
    51--63.

\bibitem[HR04]{HR}
    J.-M. Hwang, S. Ramanan,
    \textit{Hecke curves and Hitchin discriminant},
    Ann. Sci. École Norm. Sup. (4) \textbf{37}(5) (2004),
    801--817.

\bibitem[Ili94]{Ili}
    A. Iliev,
    \textit{The Fano surface of the Gushel threefold},
    Compositio Math. \textbf{94}(1) (1994),
    81--107.

\bibitem[Isk87]{Isk}
    V. A. Iskovskikh,
    \textit{On the rationality problem for conic bundles}, 
    Duke Math. J. \textbf{54}(2) (1987), 
    271--294.

\bibitem[Keb02]{Keb}
    S. Kebekus,
    \textit{Families of singular rational curves},
    J. Algebraic Geom. \textbf{11}(2) (2002),
    245--256.

\bibitem[Kol96]{Kol}
    J. Koll\'{a}r,
    Rational curves on algebraic varieties,
    Ergebnisse der Mathematik und ihrer Grenzgebiete. 3. Folge. A Series of Modern Surveys in Mathematics \textbf{32},
    Springer-Verlag, Berlin, 1996.

\bibitem[Mar82]{Mar}
    M. Maruyama,
    \textit{Elementary transformations in the theory of algebraic vector bundles},
    in Algebraic geometry (La R\'abida, 1981),
    pp. 241--266,
    Lecture Notes in Math. \textbf{961}, Springer, Berlin, 1982.

\bibitem[MM81]{MM}
    S. Mori, S. Mukai,
    \textit{Classification of Fano $3$-folds with $B_2\geq 2$},
    Manuscripta Math. \textbf{36}(2) (1981),
    147--162.

\bibitem[MM83]{MM2}
    S. Mori, S. Mukai,
    \textit{On Fano $3$-folds with $B_2\geq 2$},
    in Algebraic varieties and analytic varieties (Tokyo, 1981),
    pp. 101--129,
    Adv. Stud. Pure Math. \textbf{1}, North-Holland, Amsterdam, 1983.

\bibitem[OSW16]{OSW}
    G. Occhetta, L. E. Sol\'{a} Conde, K. Watanabe,
    \textit{Uniform families of minimal rational curves on Fano manifolds},
    Rev. Mat. Complut. \textbf{29}(2) (2016),
    423--437.

\bibitem[Pro18]{Prok}
    Yu. G. Prokhorov,
    \textit{The rationality problem for conic bundles},
    Russian Math. Surveys \textbf{73}(3) (2018),
    375--456.

\bibitem[PCS19]{PCS}
    V. V. Przyjalkowski, I. A. Cheltsov, K. A. Shramov,
    \textit{Fano threefolds with infinite automorphism groups},
    Izv. Math. \textbf{83}(4) (2019),
    860--907.

\bibitem[Sha24]{Sha}
    F. Shao,
   \textit{ On pseudoeffective thresholds and cohomology of twisted symmetric tensor fields on irreducible Hermitian symmetric spaces},
    Sci. China Math. (to appear).

\bibitem[SCW04]{SCW}
    L. E. Sol\'{a} Conde, J. A. Wi\'{s}niewski,
    \textit{On manifolds whose tangent bundle is big and 1-ample},
    Proc. London Math. Soc. (3) \textbf{89}(2) (2004),
    273--290.
    
\bibitem[SW90]{SW}
    M. Szurek, J. A. Wi\'{s}niewski,
    \textit{Fano bundles of rank 2 on surfaces},
    Compositio Math. \textbf{76}(1--2) (1990),
    295--305.

\end{thebibliography}
\end{document}